\documentclass[12pt]{article}

\usepackage{times}
\usepackage[centertags]{amsmath}
\usepackage{amssymb}
\usepackage{amsthm}
\usepackage[utf8]{inputenc}
\usepackage[T1]{fontenc}
\usepackage[colorlinks=true,urlcolor=red,citecolor=red,linkcolor=red]{hyperref}
\usepackage[a4paper, top=5cm, bottom=5cm, left=3.5cm, right=2.5cm]{geometry}
\usepackage[italian,english]{babel}
\usepackage{graphicx}
\usepackage{tikz}
\usepackage{qtree}
\usepackage{tikz-qtree}
\usepackage[lined,algonl,boxed]{algorithm2e}
\usepackage[backend=biber,hyperref]{biblatex}
\usetikzlibrary{trees,positioning,shapes} 
\usepackage{verbatim}
\newtheorem{teo}{Theorem}[section]

\newtheorem{prop}[teo]{Proposition}
\newtheorem{lem}[teo]{Lemma}
\newtheorem{cor}[teo]{Corollary}
\theoremstyle{definition}
\newtheorem{discussion}[teo]{Discussion}
\theoremstyle{definition}
\newtheorem{oss}[teo]{Remark}
\theoremstyle{definition}
\newtheorem{ex}[teo]{Example}

\newcommand{\N}{\mathbb{N}}

\newcommand{\K}{\mathbb{K}}

\newenvironment{acknowledgements}%
{\null\vfill\begin{center}%
		\bfseries Acknowledgements\end{center}}%
{\vfill\null}

\newcommand{\keywords}[1]{\emph{Keywords:} #1}
\newcommand{\MSC}[1]{\emph{Mathematics Subject Classification 2010:} #1}

\newcounter{lastnote}

\title{An algorithm for computing the Arf closure of an algebroid curve with more than one branch} 
\author{N. Maugeri\footnote{\emph{e-mail: nicola.maugeri.1992@gmail.com}}, G. Zito\footnote{\emph{e-mail: giuseppezito@hotmail.it}}}
\date{}

\begin{document}

\maketitle
\begin{abstract}
In this paper, we give a fast algorithm for the computation of  the Arf closure of an algebroid curve with more than one branch, generalizing an algorithm presented by Arslan and Sahin for the algebroid branch case.
\end{abstract}
\MSC{13A18, 14H20, 20-04.}

\section*{Introduction}
Let $R$ be an algebroid curve, i.e, following Zariski's  terminology, a one-dimensional, reduced, local ring of the form $\K[[x_1,x_2,\ldots,x_k]]/I$, where $\K$ is a field, $x_i$ are indeterminates and $I=P_1\cap\ldots\cap P_n$, with $P_i$ prime ideals of height $k-1$. If $\K$ is algebrically closed, these rings can be obtained as the completion of local rings of algebraic curves at a singular point.
Since $R$ is reduced, the ideal $I$ can be written as intersection of its minimal primes \hbox{$I=\bigcap_{i=0}^nP_i$}. Thus we can consider the inclusion
$\tau:R\hookrightarrow \K[[x_1,\ldots,x_k]]/P_1\times \ldots \times \K[[x_1,\ldots,x_k]]/P_n$.\\
Furthermore, we will set $R^i=\K[[x_1,\ldots,x_k]]/P_i$ for $i=1,\ldots n$ and these rings will be called algebroid branches. Considering the integral closures in $Q(R)$, we have $\overline{R}\cong \overline{R^1}\times \ldots \times \overline{R^k}$
where each $\overline{R^i}$ is a complete one-dimensional domain, that is a DVR, thus we can associate to each element $x\in \K[[t_i]]$ a valuation $\nu_i(x)$. Finally, because $R\subseteq \K[[t_1]]\times\K[[t_2]]\times\ldots\times\K[[t_n]]$, we can define the valuation of an element $y\in R$, such that $\tau(y)=(\phi_1(t_1),\ldots,\phi_n(t_n))$, as the vector $\nu(y)=(\nu_1(\phi_1(t_1)),\ldots,\nu_n(\phi_n(t_n)))$.\\
Now let us consider the submonoid $\nu(R)=\{\nu(x):x\in R\}\subseteq \N^n$
which is a local good semigroup (cf.\cite[pag.11]{anal:unr}).\\
We recall that a good semigroup $S$ is an Arf semigroup if $S(\alpha)-\alpha$ is a semigroup for all $\alpha\in S$, where $S(\alpha)=\{\beta\in S: \beta\geq \alpha\}$. Similarly,  $R$ is an Arf ring if $x^{-1}R(\alpha)$ is a ring for all $\alpha\in \nu(R)$, where  $R(\alpha)=\{r\in R: \nu(r)\geq \alpha\}$. Then it is possible to define the Arf closure $R^*$ of a ring $R$ as the smallest Arf ring containing $R$ (cf.\cite{Arf:multseq,Lipman:stable,Arf:characters,anal:unr}). Given an algebroid curve $R$, the semigroup $\nu(R^*)$ associated to $R^*$ is an Arf semigroup (cf.\cite[Prop 5.10]{anal:unr}).
The study of the Arf closure is motivated by its important role in the definition of the equivalence of algebroid curves. The equivalence of algebroid curves was defined in \cite{Zar:equi} as a generalization of the equivalence of algebroid branches; since two algebroid curves are equivalent if they have the same Arf closure (cf.\cite{anal:unr}), it is interesting to determine the Arf C
closure of an algebroid curve.\\
In this work, we  generalize to the case of algebroid curves the algorithm, presented by Feza Arslan and Nihl Sahin, for the computation of the Arf closure of an algebroid branch (cf.\cite{Arf:closure}) \\
In Section \ref{section1}, we  define the multiplicities of the semi-local rings in the Lipman sequence of an algebroid curve. Then, we  see how to associate to an algebroid curve a semigroup and its multiplicity tree. Furthermore, we  introduce the minimal tree of $R$, isomorphic to the multiplicity tree of $\nu(R^*)$, by associating to each node $\textbf{n}_i^j$ of the multiplicity tree an element of minimal valuation $\textbf{n}_i^j$ in  the corresponding blow up of $R$. Then we conclude the section by introducing a recursive method for the computation of the Arf closure of an algebroid curve.\\
In Section \ref{section2}, we introduce an algorithm for the computation of the multiplicity tree and the minimal tree of an algebroid curve with two branches. This algorithm will return the parametrizations of all rings in the Lipman sequence. Furthermore, we also present a method for computing the Arf closure of the algebroid curve using the information given by the algorithm (cf. Discussion \ref{D1}).\\
In Section \ref{section3} we  see how to generalize the algorithm presented in the previous section to the case of curves with an arbitrary number of branches.\\
In Section \ref{section4}, we  give a way to  improve the efficiency of our algorithm.
In particular, we   see that it is possible to compute the Arf closure of $R$ by applying the algorithm to an algebroid curve with a simpler parametrization obtained by truncating all the monomials with order bigger than the conductor of the Arf semigroup $\nu(R^*)$ (cf.Theorem \ref{T1}).
Thus, in order to determine this bound, we  need a way to estimate the conductor of  $\nu(R^*)$ directly from the parametrization of $R$. We firstly analyze the case of curves with two branches having distinct multiplicity sequences along their branches (we can recover the multiplicity sequences by  using  the algorithm of Arslan and Sahin on each branch). In this case, it is possible to find a limitation for the conductor by using only the numerical properties given by the multiplicity sequences (cf.Theorem \ref {T2}). Then, we  study the case of two-branches algebroid curves with the same multiplicity sequence on their branches. In this case, we need to work on the parametrization of $R$ to find a suitable bound (cf.Lemma \ref{L3} and Proposition \ref{P1}). We conclude by seeing how it is possible to use the bound in the two-branches case to compute a bound in the general case (cf.Remark \ref{R8}). In the end, we present an example that illustrates how the computation of the Arf closure is simplified by the truncation given by the given bound (cf.Example \ref{E4}).\\
The procedures presented here have been implemented in GAP (\cite{GAP4}).

\section{Preliminaries}\label{section1}

\subsection{The multiplicity tree of an algebroid curve $R$} \label{subs1}
If $R$ is an algebroid curve we can always associate to it a parametrization $$x_1=(\phi_{11}(t_1),\ldots, \phi_{1n}(t_n)),\ldots, x_k=(\phi_{k1}(t_1),\ldots, \phi_{kn}(t_n))$$ such that
$$ R\cong\mathbb{K}[[ \left( \phi_{11}(t_1),\ldots,\phi_{1n}(t_n)\right),\ldots, \left( \phi_{k1}(t_1),\ldots,\phi_{kn}(t_n)\right)]],$$
cf.\cite{Campillo:algebroid}.\\
In this paper, we will only consider algebroid curves given through their parametrization.\\
Since $R$ is a local ring we can define its blow-up as $Bl(R)=\cup_{i=0}^{\infty}(\mathfrak{m}^n:\mathfrak{m}^n)$, where $\mathfrak{m}$ is its maximal ideal.\\
If $R$ is an algebroid curve with maximal ideal $\mathfrak{m}=(x_1,\ldots,x_k)$, then $Bl(R)=R[x,\frac{x_1}{x},\ldots,\frac{x_k}{x}]$ (see \cite[Prop 1.1]{Lipman:stable}), where $x$ is an element of $R$ with minimal valuation (this follows from the fact that $R\subseteq \bar{R}$ is a finite integral extension on $R$). In particular if $R$ is parametrized by 
$$x_1=(\phi_{11}(t_1),\ldots, \phi_{1n}(t_n)),\ldots, x_k=(\phi_{k1}(t_1),\ldots, \phi_{kn}(t_n)),$$ 
it is easy to see that
$$Bl(R)=\mathbb{K}\left[\left[x, \frac{\left( \phi_{11}(t_1),\ldots,\phi_{1n}(t_n)\right)}{x},\ldots, \frac{\left(\phi_{k1}(t_1),\ldots,\phi_{kn}(t_n)\right)}{x}\right]\right].$$ 
In the following we denote this ring with the symbol $[x^{-1}R]$ (it is the smallest ring containing $x^{-1}R$ ).
If we consider the Lipman sequence 
$$R=R_1\subseteq R_2\subseteq R_3\subseteq\ldots,$$ where $R_i=Bl(R_{i-1})$, since $\bar{R}$ is a finite $R$-module, there exists an integer $N\in \N$ such that $R_N=\K[[t_1]]\times\ldots\times\K[[t_n]]$.\\
We know that the rings $R_i$ are semilocal rings. A semilocal ring $S \subseteq \mathbb{K}[[t_1]]\times \dots \times \mathbb{K}[[t_n]] $, parametrized by 
$$S=\mathbb{K}[[ \left( \phi_{11}(t_1),\ldots,\phi_{1n}(t_n)\right),\ldots, \left( \phi_{k1}(t_1),\ldots,\phi_{kn}(t_n)\right)]],$$  can be always seen as a product of local rings. In other words, there exists a partition $\mathfrak{P}(S)=\left\{ P_1,\ldots,P_t\right\}$ of $\left\{1,\ldots,n \right\}$, with $$ P_i=\left\{ q_{i,1},\ldots, q_{i,k(i)} \right\},$$
such that
$$S=S_1 \times \dots \times S_t,$$
where $S_i$ is a local ring contained in $\mathbb{K}[[t_{q_{i,1}}]] \times \dots \times \mathbb{K}[[t_{q_{i,k(i)}}]].$ 

We have $S=\prod_{i=1}^{t}{S(P_i)}$, where
$$ S(P_i)=\mathbb{K}\left[\left[ \left( \phi_{1q_{i,1}}(t_{q_{i,1}}),\ldots,\phi_{1q_{i,k(i)}}(t_{q_{i,k(i)}})\right),\ldots, \left( \phi_{kq_{i,1}}(t_{q_{i,1}}),\ldots,\phi_{kq_{i,k(i)}}(t_{q_{i,k(i)}})\right)\right]\right].$$
Now we need to define the multiplicity vector of a semilocal ring $S$. We have two cases

\begin{itemize}
	\item  $S$ is local.
	
	We define $\textrm{mult}(S)=\min \left\{ \nu(s): s\in S \right\}$, where $\nu$ is the valuation defined in $\mathbb{K}[[t_1]]\times \dots \times \mathbb{K}[[t_n]]$. 
	It is easy to see that  if 
	$$ S=\mathbb{K}[[ \left( \phi_{11}(t_1),\ldots,\phi_{1n}(t_n)\right),\ldots, \left( \phi_{k1}(t_1),\ldots,\phi_{kn}(t_n)\right)]],$$
	then
	$$ \textrm{mult}(S)[i]=\min\left\{ \textrm{ord}(\phi_{1i}(t_i)), \ldots, \textrm{ord}(\phi_{ki}(t_i))\right\} \textrm{ for all } i=1,\ldots,n,$$
	where with $\textrm{mult}(S)[i]$ we mean the $i$-th component of the $n$-vector $\textrm{mult}(S)$.
		
	Because the field $\mathbb{K}$ is infinite we can always find a linear combination $x_S$ of the generators of $S$, such that $\nu(x_S)=\textrm{mult}(S)$. We also set $\textrm{mult}^{*}(S)=\left\{ \textrm{mult}(S) \right\}$. Note that the multiplicity of $S$ as local ring is given by the sum of components of $\textrm{mult}(S)$.
	\item $S$ is not local
	
	Suppose that $\mathfrak{P}(S)=\left\{ P_1,\ldots,P_t\right\}$, with $$ P_i=\left\{ q_{i,1},\ldots, q_{i,k(i)} \right\},$$
	we have
	$$S=\prod_{i=1}^{t}{S(P_i)},$$
	and we already know how to compute $\textrm{mult}(S(P_i))$ because $S(P_i)$ is local in  $\mathbb{K}[[t_{q_{i,1}}]] \times \dots \times \mathbb{K}[[t_{q_{i,k(i)}}]]. $
	
	Then we can define:
	$$ \textrm{mult}^{*}(S)=\left\{ \textrm{lmult}(S(P_i)): i=1,\ldots,t \right\},$$
	
	where $ \textrm{lmult}(S(P_i))$ is an $n$-vector such that  
	
	\begin{itemize} 
		\item$ \textrm{lmult}(S(P_i))[j]=0$ if $j \notin P_i$;
		
		\item $ \textrm{lmult}(S(P_i))[q_{i,j}]=\textrm{mult}(S(P_i))[j]$, for $j=1,\ldots,k(i)$.
		
	\end{itemize} 
	To each element of $\textrm{mult}^{*}(S)$ we can associate an element of minimal value in $S$. If $x_{S(P_i)}$ is an element of minimal value on $S(P_i)$ we consider the element $x_S^i \in S$ such that 
	\begin{itemize} 
		\item$ x_S^i[j]=1$ if $j \notin P_i$;
		
		\item $ x_S^i[q_{i,j}]=x_{S(P_i)}[j]$, for $j=1,\ldots,k(i)$.
	\end{itemize}
	
\end{itemize}

Thus from the Lipman sequence of blow-ups we can recover the following sequence of subsets of $	\mathbb{N}^n:$ 
$$ \textrm{mult}^{*}(R_1), \textrm{mult}^{*}(R_2), \ldots, \textrm{mult}^{*}(R_N)=\left\{ (1,0,\ldots,0), (0,1,0,\ldots,0),\ldots,(0,\ldots,0,1)\right\}.$$
We define the multiplicity tree associated to the ring $R$, as the tree $T(R)$ with nodes in $ \cup_{i=0}^N{ \textrm{mult}^{*}(R_i)}$ and such that two nodes $v,w$ are linked if and only if there exists $m$ such that $v \in  \textrm{mult}^{*}(R_m)$ and $w \in  \textrm{mult}^{*}(R_{m+1})$
(or viceversa) and we have $ \langle v,w \rangle \neq 0$ (where with $\langle v,w \rangle$ we mean the standard scalar product in $\mathbb{N}^n$).
We can also define the minimal tree by assigning to each node $\textrm{lmult}(R_j(P_i)) $ of the multiplicity tree the corresponding element $x_{R_j(P_i)}$ of minimal value.\\ 

\subsection{The computation of the Arf closure $R^*$}
Now we want to show how the Lipman sequence can be used to compute and to give a presentation for  the Arf closure $R^*$ of $R$.\\

Arf, in his work (cf.\cite[p.267]{Arf:multseq}), showed that if $R$ is an irriducible algebroid curve of $\K[[t]]$ then 
$$R^*=\K+x\cdot Bl(R)^*,$$
where $x$ is a minimal valutation element in $R$.
As a consequence of this fact, $R^*$ can be presented as:
$$R^*=\K+\K\cdot x_1+ \K\cdot x_1 x_2+\ldots +x_1 x_2\ldots x_{N-1}\K[[t]],$$
where  $x_i$ is an element of minimal valuation in  $R_i=Bl(R_{i-1})$ (where $R=R_1$).
Now we want to adapt this computation to an algebroid curve $R \subseteq \mathbb{K}[[t_1]]\times \dots \times \mathbb{K}[[t_n]]$. We  build the Arf closure by using the following inductive process on the number of branches $n$.
\begin{itemize}
	\item Base case: $n=1$. It was proved by Arf (cf.\cite{Arf:multseq}).
	
	\item Inductive step. We suppose that we are able to solve the problem for $m<n$ and we give a solution for $n$.
	
	If $R$ is not local then there exists a partition $\mathfrak{P}(R)=\left\{ P_1,\ldots,P_t\right\}$, with $$ P_i=\left\{ q_{i,1},\ldots, q_{i,k(i)} \right\},$$
	such that
	$$R=R^1 \times \dots \times R^t,$$
	
	where $R^i$ is a local ring contained in $\mathbb{K}[[t_{q_{i,1}}]] \times \dots \times \mathbb{K}[[t_{q_{i,k(i)}}]]. $ 
	
	In this case, we have:
	$$ R^{*}=(R^1)^{*} \times \dots \times (R^t)^{*},$$
	
	and, for the inductive step, we can compute each $(R^i)^{*}$, since $k(i)<n$ for all $i$. 
	If $R=R_1$ is a local ring, using the same idea of Arf (cf.\cite[p.267]{Arf:multseq}) it is easy to see that
	$$ R_1^{*}=\mathbb{K} (1,\ldots,1)+x_1 \cdot( R_2)^{*},$$
	where $x_1$ is an element of minimal value in $R_1$ and $R_2=\textrm{Bl}(R_1)$ is the blow-up of $R_1$.
	
	If $R_2$ is local we can compute $R_2^{*}$ in the same way using $R_3=\textrm{Bl}(R_2)$ and an element of minimal value $x_2$ in $R_2$. But we know that there exist an $N$ such that $R_N$ is not local (in fact the blow-up sequence has to stabilize into $\overline{R}=\mathbb{K}[[t_1]]\times \dots \times \mathbb{K}[[t_n]]$) and therefore we are able to compute $R_N^{*}$ as we have already seen in the non-local case.
	
	Then, if we suppose that $N$ is the first integer such that $R_N$ is not local, we have
	\begin{eqnarray*}(R_1)^{*}&=& \mathbb{K} (1,\ldots,1)+x_1 \cdot( R_2)^{*}   \\ (R_2)^{*}&=& \mathbb{K} (1,\ldots,1)+x_2 \cdot( R_3)^{*} \\
		\ldots & & \ldots \\  (R_{N-1})^{*}&=& \mathbb{K} (1,\ldots,1)+x_{N-1} \cdot( R_N)^{*}, 
	\end{eqnarray*}
	and from this it follows that
	$$ (R_1)^{*}=\mathbb{K} (1,\ldots,1)+\mathbb{K}x_1+\mathbb{K} x_1x_2+\ldots+x_1 \dots x_{N-1} \cdot( R_N)^{*}.$$
	where $x_i$ is an element of minimal valuation of $R_i$.
\end{itemize}
From this procedure we see that it is important to compute the blow-up sequence $R_i$  until $R_m= \mathbb{K}[[t_1]]\times \dots \times \mathbb{K}[[t_n]]$ to understand how to compute $R^{*}$. In the following section, we will present an algorithm that gives us a way to compute this sequence along its multiplicity tree starting from a parametrization of the ring $R$.
\\
\subsection{The properties of the multiplicity tree of an Arf good semigroup}
We can notice that from the previous construction for the Arf closure $R^*$, it easily follows that the  multiplicity tree of $R$ defined above is nothing but the multiplicity tree of the Arf good semigroup $\nu(R^*)$ (cf.\cite{Arf:multseq,Lipman:stable,Arf:characters,anal:unr}). Thus, it is useful to recall some properties of this kind of objects. Given an Arf semigroup $S\subseteq \mathbb{N}^n$, its multiplicity tree is a tree $T$ such that the nodes are vector $ \textbf{n}_i^j \in \mathbb{N}^n$, where with $\textbf{n}_i^j$ we mean that  this node is in the $i$-th branch on the $j$-th level (the root of the tree is $\textbf{n}_{1}^1=\textbf{n}_i^1 $ for all $i$ because we are in the local case and at level one all the branches must be glued) and we have 

$$ S=\left\{\textbf{0}\right\} \bigcup_{T'} \left\{ \sum_{\textbf{n}_i^j \in T' } {\textbf{n}_i^j}\right\},$$

where $T'$ ranges over all finite subtree of $T$ rooted in   $\textbf{n}_1^1$.

Furthermore a tree $T$ is a multiplicity tree of an Arf semigroup if and only if its  nodes  satisfy the following properties (see. \cite[Thm 5.11]{anal:unr}).

\begin{itemize}
	\item there exists $L \in \mathbb{N}$ such that for $m \geq L$, $\textbf{n}_i^m=(0,\ldots,0,1,0\ldots,0)$ (the nonzero coordinate is in the $i$-th position) for any $i=1,\ldots,n$;
	\item $\textbf{n}_i^j[h]=0$ if and only if $\textbf{n}_i^j$ is not in the $h$-th branch of the tree;
	\item each $\textbf{n}_i^j$ can be obtained as a sum of nodes in a finite subtree $T'$ of $T$ rooted in $\textbf{n}_i^j$.
\end{itemize}
Notice that from these properties it follows that we must have multiplicity sequences along each branch.
\\ \\
Suppose now that  $E$ is an ordered collection of $n$ multiplicity sequences (that will be the multiplicity branches of a multiplicity tree). Since any multiplicity sequence is a nonincreasing sequence of integers that stabilizes to 1, we can describe them by the vectors
$$\textbf{M}_i=[M_i[1],\ldots, M_i[l_i]],$$
where if $\textbf{M}_i \neq [1,1\ldots]$,  $l_i=\max\{j\hspace{0.1cm}:\hspace{0.1cm}M_i[j]\neq 1\}$ will be said length of $i$-th sequence. For $\textbf{M}_i=[1,1,\ldots]$ we will set $l_i=1$ by definition.

Denote by $\tau(E)$ the set of all multiplicity trees having the $n$ branches in $E$.

We give now a way to describe a tree of $\tau(E)$.
If $T \in \tau(E)$, it can be represented by an upper triangular matrix $ n \times n$  $$M(T)_{E}=\left(  \begin{matrix} 0 & p_{1,2} & p_{1,3} & \ldots & p_{1,n} \\ 
0 & 0 & p_{2,3} & \ldots & p_{2,n} \\ \ldots & \ldots & \ldots & \ldots & \ldots \\ 0 & 0 & 0 & \ldots& p_{n-1,n} \\ 0 & 0 & 0 & \ldots & 0 \\  \end{matrix}\right),$$ where $p_{i,j}$ is the highest level such that the  $i$-th and the $j$-th branches are glued in $T$.

\section{The algorithm in the two-branches case}
\label{section2}
In this section we give an algorithm for the computation of the Arf closure of an algebroid curve that has the following parametrization:
$$ R=\mathbb{K}[[\left(\phi_1(t),\psi_1(u)\right),\ldots, \left(\phi_n(t),\psi_n(u)\right)]].$$

Notice that, in order to lighten the notation, we are focusing on the two-branches case. However, the algorithm can be easily adapted to the general case with small modifications as we will see in section \ref{section3}.

First of all we fix some notations.
In the following we will always assume that a parametrization does not contain an element $y=(\phi(t),\psi(u))$ such that $\textrm{ord}(\phi(t))=\textrm{ord}(\psi(u))=0$ and with  $\phi(0)=\psi(0)$. If, in the following constructions, we will produce a parametrization that contains such an element, we will always convert it to $\overline{y}=y-(\phi(0),\psi(0))$ (it is possible to do that because $(\phi(0),\psi(0))$ is a multiple of the unit vector).
For each $m \geq 0 $ we will denote by
$$ R_m=\mathbb{K}\left[\left[\left(\phi_1^{(m)}(t),\psi_1^{(m)}(u)\right),\ldots, \left(\phi_{n(m)}^{(m)}(t),\psi_{n(m)}^{(m)}(u)\right)\right]\right],$$
the parametrization of the $m$-th blow-up of $R$ (we put by definition $R_1=R$).
Furthermore, if $R_m$ is local, we denote by 
$ \textrm{mult}(R_m)=\min\left\{ \nu(r): r \in R_m \right\}$, the multiplicity vector of the $m$-th blow-up, where $\nu$ is the valutation defined in $R$. With our notation it is easy to see that we have
$$ \textrm{mult}(R_m)=\left(\min\left\{ \textrm{ord}(\phi_i^{(m)}(t)),i=1,\ldots,n(m) \right\},\min\left\{ \textrm{ord}(\psi_i^{(m)}(u)),i=1,\ldots,n(m) \right\}\right).$$
Finally, always with the assumption that $R_m$ is local, we denote by $x_{R_m}$ an element of $R_m$ with valutation $\textrm{mult}(R_m)$.

\begin{oss}
	\label{R1}
	For the choice of the element $x_{R_m}$  we can always consider either one of the $\left(\phi_i^{(m)}(t),\psi_i^{(m)}(u)\right)$ or the sum of two of them. To see it we denote by $y_i=\left(\phi_i^{(m)}(t),\psi_i^{(m)}(u)\right)$ for $i=1, \ldots,n(m)$.
	If there exists  $y_i$  in the parametrization such that $\textrm{mult}(R_m)=\nu(y_i)$ we can set $x_{R_m}=y_i$.
	Otherwise, for the definition of $\textrm{mult}(R_m)$ there must exist $i,j$ with $i\neq j$ such that 
	$$ \left(\textrm{ord}(\phi_i^{(m)}(t)),\textrm{ord}(\psi_j^{(m)}(u))\right) =\textrm{mult}(R_m),$$
	then $y_i+y_j$ is a good choice for $x_{R_m}$ (in this case order cancellations cannot happen).	
\end{oss}

The following lemma will help us to understand when a $R_m$ is not local from its parametrization.

\begin{lem}
	\label{L1}
	Consider $$ R=\mathbb{K}[[\left(\phi_1(t),\psi_1(u)\right),\ldots, \left(\phi_n(t),\psi_n(u)\right)]].$$
	We have that 
	$$ R=\mathbb{K}[[ \phi_1(t),\ldots,\phi_n(t)]] \times \mathbb{K}[[ \psi_1(u),\ldots,\psi_n(u)]]$$ if and only if  at least one of the following two conditions holds:
	\begin{itemize}
		\item There exists $\left(\phi_i(t),\psi_i(u)\right)$ in the parametrization such that
		$$ \textrm{ord}(\phi_i(t))\cdot \textrm{ord}(\psi_i(t))=0   \textrm{ and } {\textrm{ord}(\phi_i(t))}^2+ {\textrm{ord}(\psi_i(t))}^2\neq 0; $$
		\item  There exists $y=\left(\phi_i(t),\psi_i(u)\right)$ in the parametrization such that
		$$ \nu(y)=(0,0) \textrm{ and } \phi_i(0) \neq \psi_i(0).$$
	\end{itemize}
\end{lem}
\begin{proof}
	($\Leftarrow$). 
	Let us suppose that the first condition holds.
	Without loss of generality we can suppose that the element  $y=\left(\phi_1(t),\psi_1(u)\right)$ in the parametrization is such that
	$\textrm{ord}(\phi_1(t))=0$ and $\textrm{ord}(\psi_1(u)) \neq 0$.
	Then we have $\phi_1(0) \neq 0$.  Therefore $\phi_1(t)$ is invertible in $\mathbb{K}[[\phi_1(t)]]$ because its inverse is
	$$ (\phi_1(t))^{-1}=(\phi_1(0))^{-1}\cdot \sum_{i=0}^{+\infty}{(-1)^i\left(\frac{\phi_1(t)-\phi_1(0)}{\phi_1(0)}\right)^i}. $$
	Thus in $\mathbb{K}[[y]] \subseteq R$ there exists an element of the form $ z=\left( (\phi_1(t))^{-1}, g(u) \right) $. Then  we have $$ R \ni y \cdot z= \left( 1, \psi_1(u)\cdot g(u)\right)=\left(1,h(u)\right),  $$
	where $\textrm{ord}(h(u))>0$.
	But $(1,1) \in R $ so  $\left(1,h(u)\right)-(1,1)=(0,-1+h(u))$ belongs to $R$.
	Now, $h(u) \in \mathbb{K}[[\psi_1(u)]]$ and therefore $-1+h(u)$ is invertible in this ring. From this it follows again that there exist an element of the type $\left(l(t),(-1+h(u))^{-1} \right) \in R$ and we have:
	
	$$ R \ni (0,-1+h(u)) \cdot  \left(l(t),(-1+h(u))^{-1} \right)=(0,1) \Rightarrow (1,1)-(0,1)=(1,0) \in R.$$
	
	Finally we obtain that 
	
	$$ \mathbb{K}[[ \phi_1(t),\ldots,\phi_n(t)]] \times \left\{0\right\}=(1,0) \cdot R \subseteq R, $$ 
	$$ \left\{0\right\} \times \mathbb{K}[[ \psi_1(u),\ldots,\psi_n(u)]]=(0,1) \cdot R \subseteq R,$$
	therefore we have 
	$\mathbb{K}[[ \phi_1(t),\ldots,\phi_n(t)]] \times \mathbb{K}[[ \psi_1(u),\ldots,\psi_n(u)]] \subseteq R$ and because the 
	inverse containment is trivial we have our thesis. 
	Suppose now that the second condition holds. Let us consider $y=\left(\phi_i(t),\psi_i(u)\right)$ in the parametrization such that
	$$ \nu(y)=(0,0) \textrm{ and } \phi_i(0) \neq \psi_i(0).$$
	Thus if we consider $(\phi_i(0),\phi_i(0)) \in R $ we have that $y-(\phi_i(0),\phi_i(0)) \in R$ is an element that fulfills  the first condition and we can use the same arguments of the first part of the proof.\\
	($\Rightarrow$). It is trivial, in fact if we suppose by contradiction that in the parametrization does not appear elements that fulfill the condition of the theorem then it would easily follow that in $R$ we cannot find an element $(\phi(t),\psi(u))$ such that $\phi(t)$ is invertible and $\psi(u)$ is not invertible and this is absurd for the hypotheses on $R$.
\end{proof}

\begin{oss}
	\label{R2}
	
	If we have a ring $S$ such that $$S=\mathbb{K}[[ \phi_1(t),\ldots,\phi_n(t)]] \times \mathbb{K}[[ \psi_1(u),\ldots,\psi_n(u)]],$$
	then $S$ is not local and, following the notations of the first section, we have that
	$$ \textrm{mult}^*(S)=\left\{ (m_1,0), (0,m_2) \right\},$$
	where $ m_1$ is the multiplicity of the algebroid branch associated to  $S_1=\mathbb{K}[[ \phi_1(t),\ldots,\phi_n(t)]]$ and 
	$ m_2$ is the multiplicity of the algebroid branch associated to $S_2=\mathbb{K}[[ \psi_1(u),\ldots,\psi_n(u)]]$.
	It is easy to show that we have
	\begin{itemize}
		\item $m_1=\min \left\{ \textrm{ord}(\phi_i(t)-\phi_i(0)): i=1,\ldots,n \right\}$;
		\item $m_2=\min \left\{ \textrm{ord}(\psi_i(u)-\psi_i(0)): i=1,\ldots,n \right\}$.
	\end{itemize}
	Then we can denote by $x_S^1$ an element of $S_1$ with order $m_1$ and by   $x_S^2$ an element of $S_2$ with order $m_2$.
	It is clear  that there exist $i,j$ such that $x_S^1=\phi_i(t)-\phi_i(0)$ and   $x_S^2=\psi_j(u)-\psi_j(0)$.
\end{oss}

Now we want to develope an algorithm for the computation of the Arf closure  $R^*$ of $R$.
As we have seen in the previous section, we need to compute the blow-up chain $R_m$ of $R$ in order to find the multiplicity tree of $R^{*}$. In particular we have to find an integer $N$ such that $R_N=\mathbb{K}[[t]] \times \mathbb{K}[[u]]$.
From the properties of the ring of formal power series this is equivalent to find an $N$ such that $R_N$ is not local and such that $$ \textrm{mult}^*(R_N)=\left\{ (1,0), (0,1) \right\}.$$
Taking in account the definitions given at the beginning of this section we can consider the following algorithm.

\begin{algorithm}
	\SetKwData{Left}{left}
	\SetKwData{This}{this}
	\SetKwData{Up}{up}
	\SetKwFunction{Union}{Union}
	\SetKwFunction{FindCompress}{FindCompress}
	\SetKwInOut{Input}{input}
	\SetKwInOut{Output}{output}
	\caption{}
	\Input{$ R=\mathbb{K}[[\left(\phi_1(t),\psi_1(u)\right),\ldots, \left(\phi_n(t),\psi_n(u)\right)]]$}
	\Output{The sequence $R_m$ of blow-ups of $R$ until $R_m=\mathbb{K}[[t]] \times \mathbb{K}[[u]]$}
	\BlankLine
	$ m \longleftarrow 1$
	
	$ R_1 \longleftarrow R$
	
	\While{$\textrm{mult}^*(R_m) \neq \left\{ (1,0), (0,1) \right\}$}{ \If{$ R_m$ is local}{$ m \longleftarrow m+1$ 
			
			$ R_m \longleftarrow [(x_{R_{m-1}})^{-1} R_{m-1}]$}
		
		\If{$ R_m=R_m^1 \times R_m^2$ is not local}{$ m \longleftarrow m+1$ 
			
			$ R_m \longleftarrow \left[(x^1_{R_{m-1}})^{-1} R_{m-1}^1\right] \times \left[(x^2_{R_{m-1}})^{-1} R_{m-1}^2\right]  $}}

	\Return $R_1,R_2,\ldots,R_m$
	\label{algo_disjdecomp}
\end{algorithm}

The algorithm produces the blow-up chain because we know  that in the local case we have $R_m=\left[ (x_{R_{m-1}})^{-1} R_{m-1}\right]$ and  we have seen in the previous section that a parametrization for $R_m$ is therefore given by $$ R_m =\mathbb{K}\left[\left[ \frac{\left(\phi_1^{(m-1)}(t),\psi_1^{(m-1)}(u)\right)}{x_{R_{m-1}}},\ldots, \frac{\left(\phi_{n}^{(m-1)}(t),\psi_{n}^{(m-1)}(u)\right)}{x_{R_{m-1}}},x_{R_{m-1}}\right]\right].$$

On the other hand, if $R_{m-1}$ is not local we have that 

$$ R_{m-1}= \mathbb{K}[[\phi_1^{(m-1)}(t),\ldots,\phi_{n(m-1)}^{(m-1)}(t)]] \times  \mathbb{K}[[\psi_1^{(m-1)}(u),\ldots,\psi_{n(m-1)}^{(m-1)}(u)]], $$

therefore in order to find $R_m$ we have to apply the algorithm of Arslan-Sahin to each component of the cartesian product finding
$ R_m =\left[(x^1_{R_{m-1}})^{-1} R_{m-1}^1\right] \times \left[(x^2_{R_{m-1}})^{-1} R_{m-1}^2\right]  $ which can be computed as
$$  R_{m}= \mathbb{K}\left[\left[\frac{\phi_1^{(m-1)}(t)}{x_{R_{m-1}}^1},\ldots,\frac{\phi_{n(m-1)}^{(m-1)}(t)}{x_{R_{m-1}}^1},x_{R_{m-1}}^1\right]\right] \times \mathbb{K}\left[\left[\frac{\psi_1^{(m-1)}(u)}{x_{R_{m-1}}^2},\ldots,\frac{\psi_{n(m-1)}^{(m-1)}(u)}{x_{R_{m-1}}^2},x_{R_{m-1}}^2\right]\right]. $$
So, because at each step we know a parametrization for the $m$-th blow-up we have a way to compute the $m+1$-th and we can stop when we reach $\textrm{mult}^*(R_m) = \left\{ (1,0), (0,1) \right\}$.

\begin{oss}
	\label{R3}
	In the previous algorithm,we divide by an element of minimal valuation, considering element of the type $\displaystyle \frac{(\phi(t),\psi(u))}{x}$. It is convenient to work with such an element as a fraction (cancelling if possible the common factors between the numerator and the denominator) . In this way we can still express it by a finite set of information avoiding the problem of expanding it in power series.
\end{oss}

When the algorithm stops, we are able to build the multiplicity tree $T$ of $R^*$. It will be a multiplicity tree of an Arf semigroup of $\mathbb{N}^2$, therefore it can be represented by a collection $E=\left\{M_1,M_2\right\}$ of two multiplicity sequences and  an integer $p_1$, where $p_1$ is the highest level where the two branches of $T$ are still glued.
To find $p_1$ we have to check the first $m$ such that, in our algorithm, we obtain that $R_m$ is not local. Then we have $p_1=m-1$.

Furthermore, if $R_1=R,R_2,\ldots,R_m$ is the output of the algorithm we have that:
$$M_1[i]=\textrm{mult}(R_i)[1] \textrm{ for } i=1,\ldots,p_1 \textrm{ and }M_1[i]=(\textrm{mult}^*(R_i)[1])[1] \textrm{ for } i=p_1+1,\ldots,m;$$
$$M_2[i]=\textrm{mult}(R_i)[2] \textrm{ for } i=1,\ldots,p_1 \textrm{ and }M_2[i]=(\textrm{mult}^*(R_i)[2])[2] \textrm{ for } i=p_1+1,\ldots,m.$$

\begin{oss}
	\label{R4}
	The multiplicity sequences $M_1$ and $M_2$ can be also found by using the algorithm of Arslan and Sahin to the algebroid branches given by the parametrizations 
	$$ R^1=\mathbb{K}[[ \phi_1(t),\ldots,\phi_n(t)]]  \textrm{ and } R^2= \mathbb{K}[[ \psi_1(u),\ldots,\psi_n(u)]].$$ 
\end{oss}

In the following image we have the multiplicity tree and the minimal tree of $R^{*}$.\\
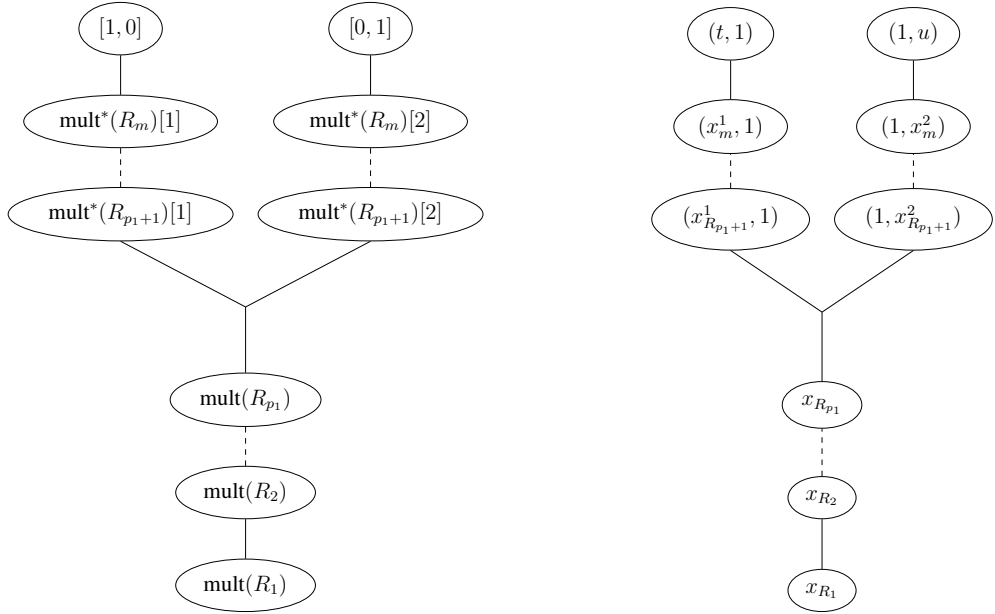
\begin{center}
	\begin{tikzpicture}[grow'=up,sibling distance=13pt,scale=.70]
	\tikzset{level distance=50pt,every tree node/.style={draw,ellipse}} \Tree [ .$\textrm{mult}(R_1)$ [ .$\textrm{mult}(R_2)$ \edge[dashed]; [ .$\textrm{mult}(R_{p_1})$ [ [ .$\textrm{mult}^*(R_{p_1+1})[1]$ \edge[dashed]; [.$\textrm{mult}^*(R_{m})[1]$ [.$[1,0]$ ]]] [ .$\textrm{mult}^*(R_{p_1+1})[2]$ \edge[dashed]; [.$\textrm{mult}^*(R_{m})[2]$ [.$[0,1]$ ]]] ] ] ]]  \end{tikzpicture} 
	\hspace{2cm}
	\begin{tikzpicture}[grow'=up,sibling distance=13pt,scale=.70]
	\tikzset{level distance=50pt,every tree node/.style={draw,ellipse}} \Tree [ .$x_{R_1}$ [ .$x_{R_2}$ \edge[dashed]; [ .$x_{R_{p_1}}$ [ [ .$(x_{R_{p_1+1}}^1,1)$ \edge[dashed]; [.$(x_m^1,1)$ [.$(t,1)$ ]]] [ .$(1,x_{R_{p_1+1}}^2)$ \edge[dashed]; [.$(1,x_m^2)$ [.$(1,u)$ ]]] ] ] ]]  \end{tikzpicture} 
\end{center}

Notice that the algorithm computes all the tools needed to construct the previous two trees.
If the tree $T$ of $R^*$ is represented by the matrix $M(T)_E= \left(\begin{matrix}
0 & p_1 \\ 0 & 0
\end{matrix}\right)$ with $E=\left\{M_1,M_2\right\}$, the conductor of the associated Arf semigroup is $c=(c[1],c[2])$ with $$c[i]=\sum_{k=1}^{\max(l_i,p_1)}{M_i[k]},$$ where $l_i$ is the length of the multiplicity sequence $M_i$.\\
We have that $(t^{c[1]},u^{c[2]})\cdot \left(\mathbb{K}[t] \times \mathbb{K}[u] \right) \subseteq R^{*}$.
\begin{discussion}
\label{D1}
Now we want to find a method to compute the Arf closure through a presentation. In the previous section, we have seen  how to construct it recursively. In the two-branches case we have that:

\begin{align*}
&R_i^*=\K(1,1)+x_{R_i}R_{i+1}^* \hspace{0.5cm} \textrm{for } i=1,\ldots,p_1\\
&R_{p_1+1}^*=(R_{p_1+1}^1)^*\times(R_{p_1+1}^2)^*\\
&R_i^1=\K[[t]] \hspace{0.5cm} \textrm{for } i>\max\{l_1,p_1\}\\
&R_i^2=\K[[u]] \hspace{0.45cm} \textrm{for } i>\max\{l_2,p_1\}\\
\end{align*}
\vspace{0.5cm}
and 
$$\textrm{if }\max\{l_j,p_1\}>p_1 \hspace{0.5cm} (R_i^{j})^*=\K+x_{R_i}^j(R_{i+1}^{j})^*\hspace{0.7cm}\textrm{for } i=p_1+1,\ldots,\max(l_j,p_1);\hspace{0.1cm}j=1,2$$
If we denote by $d_j=\max(l_j,p_1)$, by substituting the expression in the reverse order we have that:

{\small
	\begin{align*}
	&(R_{d_1}^{1})^*=\K+x_{R_{d_1}}^1\K[[t]];\hspace{3.5cm} (R_{d_2}^{2})^*=\K+x_{R_{d_2}}^2\K[[u]]; \\
	&(R_{d_1-1}^{1})^*=\K+x_{R_{d_1-1}}^1\K+x_{R_{d_1-1}}^1x_{R_{d_1}}^1\K[[t]];\hspace{0.65cm} 
	(R_{d_2-1}^{2})^*=\K+x_{R_{d_2-1}}^2\K+x_{R_{d_2-1}}^2x_{R_{d_2}}^2\K[[u]];\\
	&\ldots\ldots\ldots\ldots\ldots\ldots\ldots\ldots\\
	&\ldots\ldots\ldots\ldots\ldots\ldots\ldots\ldots\\
	&(R_{p_1+1}^{1})^*=\K+x_{R_{p_1+1}}^1\K+x_{R_{p_1+1}}^1x_{R_{p_1+2}}^1\K+\ldots+x_{R_{p_1+1}}^1x_{R_{p_1+2}}^1\ldots x_{R_{d_1}}^1\K[[t]];\\
	&(R_{p_1+1}^{2})^*=\K+x_{R_{p_1+1}}^2\K+x_{R_{p_1+1}}^2x_{R_{p_1+2}}^2\K+\ldots+x_{R_{p_1+1}}^2x_{R_{p_1+2}}^2\ldots x_{R_{d_2}}^2\K[[u]];
	\end{align*}
}
and
{\small
	\begin{align*}
	&R_{p_1+1}^*=(R_{p_1+1}^{1})^*\times (R_{p_1+1}^{2})^*  \\
	&R_{p_1}^*=\K(1,1)+ x_{R_{p_1}}((R_{p_1+1}^{1})^*\times (R_{p_1+1}^{2})^*)\\
	&\ldots\ldots\ldots\ldots\ldots\ldots\ldots\ldots\\
	&R^*=\K(1,1)+ x_{R_1}\K+\ldots+ x_{R_{p_1}}x_{R_{p_1-1}}\ldots x_{R_1}((R_{p_1+1}^{1})^*\times (R_{p_1+1}^{2})^*)
	\end{align*}
}
Finally, comparing last two relations, we obtain
{\small
	\begin{align*}
	&R^*=\K(1,1)+ x_{R_1}\K+\ldots+\\
	&+x_{R_{p_1}}x_{R_{p_1-1}}\ldots x_{R_1}\left[(\K+\ldots+x_{R_{p_1+1}}^1\ldots x_{R_{d_1}}^1\K[[t]])\times(\K+\ldots+x_{R_{p_1+1}}^2\ldots x_{R_{d_2}}^2\K[[u]])\right].
	\end{align*}
}
Developing the Cartesian product, we find:
{\small
	\begin{align*}
	&R^*=\K(1,1)+ x_{R_1}\K +\dots+x_{R_{p_1}}\ldots x_{R_1} \K+ x_{R_{p_1}}\ldots x_{R_1}(1,x_{R_{p_1+1}}^2)\K+ \ldots +\\ &+x_{R_{p_1}}\dots x_{R_1}(1,x_{R_{p_1+1}}^2\dots x_{R_{d_2}}^2)(\K \times \K[[u]])
	+x_{R_{p_1}}\ldots x_{R_1}(x_{R_{p_1+1}}^1,1)\K+\ldots+ \\ &+x_{R_{p_1}}\dots x_{R_1}(x_{R_{p_1+1}}^1,x_{R_{p_1+1}}^2\dots x_{R_{d_2}}^2)(\K\times \K[[u]])+ \ldots +\\
	&+ x_{R_{p_1}}\dots x_{R_1} (x_{R_{p_1+1}}^1\ldots x_{R_{d_1}}^1,1)(\K[[t]]\times \K)+\ldots+(t^{c[1]},u^{c[2]})\cdot(\K[[t]]\times\K[[u]]),
	\end{align*}}
because $$x_{R_{p_1}}\dots x_{R_1}(x_{R_{p_1+1}}^1\ldots x_{R_{d_1}}^1,x_{R_{p_1+1}}^2\dots x_{R_{d_2}}^2)\cdot(\K[[t]]\times\K[[u]])=(t^{c[1]},u^{c[2]})\cdot(\K[[t]]\times\K[[u]]).$$
Notice that the elements with valuation greater than the conductor can be erased.
We observe that the elements in the expression have all different valuation and each of them has valuation corresponding to an element in $\nu(R)$ that is not greater than the conductor.\\
The elements with valuation not smaller than the conductor have to belong to the set $(t^{c[1]},y)\cdot(\K[[t]]\times\K)$ with $ord(y)<c[2]$ or $(z,u^{c[2]})\cdot(\K\times\K[[u]] )$ with $ord(z)<c[1]$.\\
Each element of the set $(t^{c[1]},y)\cdot(\K[[t]]\times\K)$ can be written as a sum of an element in $(0,y)\K$ and an element of $(t^{c[1]},u^{c[2]})\cdot(\K[[t]]\times\K[[u]])$. Similarly each element of the set $(z,u^{c[2]})\cdot(\K\times\K[[u]])$ can be written as a sum of an element in $(z,0)\K$ and an element of $(t^{c[1]},u^{c[2]})\cdot(\K[[t]]\times\K[[u]])$.\\
If we define
$$Y^0=\{(y,z)\in R^*: v((y,z))<c\}$$
$$Y^1=\{(0,y)\in R^*: ord(y)<c[2]\}$$
$$Y^2=\{(z,0)\in R^*: ord(z)<c[1]\}$$
$$Y:=Y_0\cup Y_1 \cup Y_2,$$
we have a presentation of the type
$$R^{*}=\mathbb{K}(1,1)+\mathbb{K}y_1+\dots+\mathbb{K}y_k+(t^{c[1]},u^{c[2]})\cdot \left( \mathbb{K}[[t]] \times \mathbb{K}[[u]] \right), $$
where the elements $y_i$ belong to $Y$ and we have one and only one representative for each valuation not greater than the conductor. Now we will show how to compute these elements from our algorithm output.

We define small elements of a good semigroup $S$ the elements of the semigroup that are smaller than  or equal to the conductor.
Thus from the properties of the multiplicity tree of an Arf semigroup, it follows that an element $v$ of $\textrm{Small}(\nu(R^*))$ can be obtained as the sum of the nodes of a subtree of $T(R)$ rooted in $\textrm{mult}(R)$ and contained in the subtree that gives the conductor.\\
Then it is easy to find an element $y$ with valuation $v$. It suffices to consider the corresponding subtree in the minimal tree of $R^{*}$ and multiply all its nodes.
We suppose that $s_1,\ldots,s_k$ are the elements of $R^{*}$ such that
$$ \left\{\nu(s_1),\ldots,\nu(s_k),c \right\}=\textrm{Small}(\nu(R^*)), $$
if we consider the elements  $\overline{s_1},\ldots,\overline{s_k}$, obtained by truncating the monomials of degree bigger that the corresponding component of the  conductor, it is easy to see that they are the elements $y_i$ that we were searching.

\end{discussion}

\begin{ex}
	\label{E1}
	Consider $$ R=R_1=\mathbb{K}[[(t^5+t^{10},u^7),(t^8,u^{11}+u^{13})]].$$
	
	We have $\textrm{mult}(R_1)=(5,7)$. We can choose $x_{R_1}=(t^5+t^{10},u^7)$ as an element of minimal value in $R_1$. Therefore we have
	$$R_2=\mathbb{K}\left[\left[x_{R_1}=(t^5+t^{10},u^7),\frac{(t^8,u^{11}+u^{13})}{x_{R_1}}\right]\right]=\mathbb{K}\left[\left[(t^5+t^{10},u^7),\left(\frac{t^3}{1+t^5},u^4+u^6\right)\right]\right].$$
	
	$R_2$ is still local and we have $\textrm{mult}(R_2)=(3,4)$. We can choose $x_{R_2}=\left(\frac{t^3}{1+t^5},u^4+u^6\right)$. Thus we have 
	$$R_3=\mathbb{K}\left[\left[x_{R_2}=\left(\frac{t^3}{1+t^5},u^4+u^6\right),\frac{(t^5+t^{10},u^7)}{x_{R_2}}\right]\right]=$$ $$=\mathbb{K}\left[\left[\left(\frac{t^3}{1+t^5},u^4+u^6\right),\left(t^2(1+t^5)^2,\frac{u^3}{1+u^2}\right)\right]\right].$$
	
	$R_3$ is still local and we have $\textrm{mult}(R_3)=(2,3)$. We can choose $x_{R_3}=\left(t^2(1+t^5)^2,\frac{u^3}{1+u^2}\right)$. Thus we have 
	$$ R_4=\mathbb{K}\left[\left[x_{R_3}=\left(t^2(1+t^5)^2,\frac{u^3}{1+u^2}\right), \left( \frac{t}{(1+t^5)^3},u(1+u^2)^2\right)\right]\right].$$
	$R_4$ is still local and we have $\textrm{mult}(R_4)=(1,1)$. We can choose $x_{R_4}= \left( \frac{t}{(1+t^5)^3},u(1+u^2)^2\right)$. Thus we have 
	$$ R_5=\mathbb{K}\left[\left[x_{R_4}= \left( \frac{t}{(1+t^5)^3},u(1+u^2)^2\right), \left( t(1+t^5)^5,\frac{u^2}{(1+u^2)^3}\right)\right]\right].$$
	$R_5$ is still local and we have $\textrm{mult}(R_5)=(1,1)$. We can choose again $x_{R_5}= \left( \frac{t}{(1+t^5)^3},u(1+u^2)^2\right)$. Thus we have 
	$$ R_6=\mathbb{K}\left[\left[x_{R_5}= \left( \frac{t}{(1+t^5)^3},u(1+u^2)^2\right), \left((1+t^5)^8,\frac{u}{(1+u^2)^5}\right)\right]\right].$$
	This time, for the Lemma \ref{L1}, we have that $R_6$ is not local because we have the element $ \left((1+t^5)^8,\frac{u}{(1+u^2)^5}\right)$ with valuation $(0,1)$. We can write:
	$$ R_6=\mathbb{K}\left[\left[\frac{t}{(1+t^5)^3},(1+t^5)^8\right]\right] \times \mathbb{K}\left[\left[u(1+u^2)^2,\frac{u}{(1+u^2)^5}\right]\right]=K[[t]] \times K[[u]].$$
	Thus we have $\textrm{mult}^*(R_6)=\left\{ (1,0),(0,1)\right\}$, and we can stop the algorithm.
	Then the multiplicity tree of $R^{*}$ and the minimal tree are:
	
	
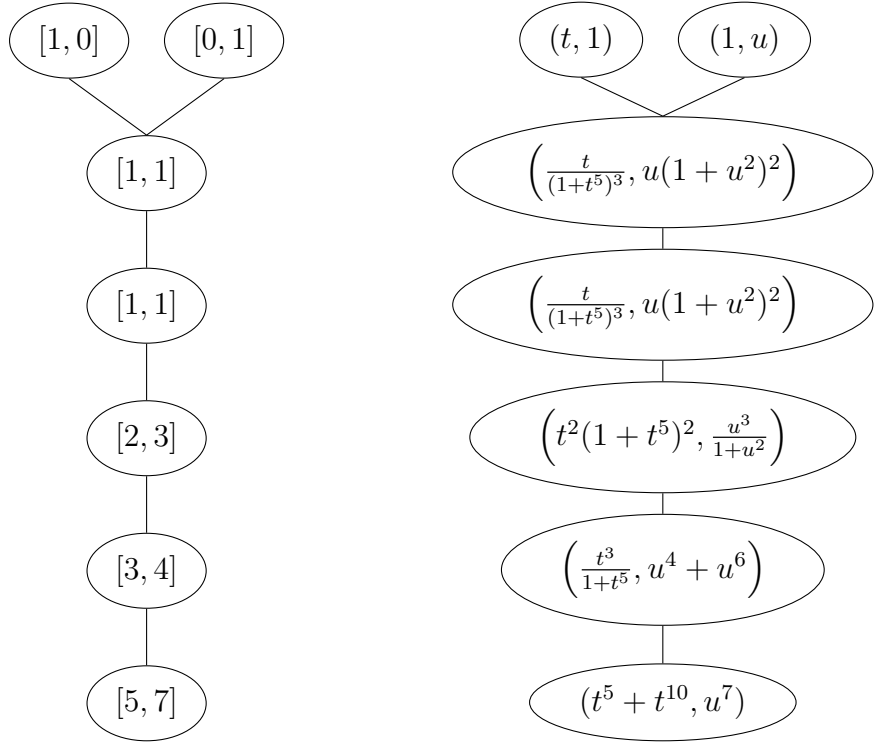
\begin{center}
		\begin{tikzpicture}[grow'=up,sibling distance=13pt]
		\tikzset{level distance=50pt,every tree node/.style={draw,ellipse}} \Tree [ .$[5,7]$ [ .$[3,4]$ [ .$[2,3]$ [ .$[1,1]$ [.$[1,1]$ [.$[1,0]$ ] [.$[0,1]$ ] ]]]]]  \end{tikzpicture} 
		\hspace{2cm}
		\begin{tikzpicture}[grow'=up,sibling distance=13pt]
		\tikzset{level distance=50pt,every tree node/.style={draw,ellipse}} \Tree [ .$\left(t^5+t^{10},u^7\right)$ [ .$\left(\frac{t^3}{1+t^5},u^4+u^6\right)$ [ .$\left(t^2(1+t^5)^2,\frac{u^3}{1+u^2}\right)$ [ .$\left(\frac{t}{(1+t^5)^3},u(1+u^2)^2\right)$ [.$\left(\frac{t}{(1+t^5)^3},u(1+u^2)^2\right)$ [.$(t,1)$ ] [.$(1,u)$ ] ]]]]]  \end{tikzpicture}  
	\end{center}

	The multiplicity tree $T$ is  $M(T)_E= \left(\begin{matrix}
	0 & 5 \\ 0 & 0
	\end{matrix}\right)$ where $E=\left\{M_1=[5,3,2], M_2=[7,4,3] \right\}$. We can easily see that conductor $c$ of $\nu(R^*)$ is $c=(12,16)$. We can also compute $\textrm{Small}(\nu(R^*))$ finding that	
	$$ \textrm{Small}(\nu(R^*))=\left\{ (5,7), (8,11 ), (10,14), (11,15), (12, 16) \right\}.$$ 
	Considering the expression of the elements of $\textrm{Small}(\nu(R^*))$ as a sum of nodes in a subtree of $T$ we can produce the following elements of $R^{*}$ as product of the corresponding nodes on  the minimal tree of $R^{*}$:
	$$ \left\{ (t^5+t^{10},u^7), (t^8,u^{11}+u^{13}), (t^{10}(1+t^5)^2,u^{14}), \left(\frac{t^{11}}{1+t^5},u^{15}(1+u^2)^2\right), (t^{12},u^{16}) \right\}.$$
	
	Finally we have $$ R^{*}=\mathbb{K}(1,1)+\mathbb{K}  (t^5+t^{10},u^7)+\mathbb{K}(t^8,u^{11}+u^{13})+\mathbb{K}(t^{10}(1+t^5)^2,u^{14})+\mathbb{K}\left(\frac{t^{11}}{1+t^5},u^{15}(1+u^2)^2\right)+$$ $$+ (t^{12},u^{16}) \left(\mathbb{K}[[t]] \times \mathbb{K}[[u]] \right)=\mathbb{K}(1,1)+\mathbb{K}  (t^5+t^{10},u^7)+\mathbb{K}(t^8,u^{11}+u^{13})+\mathbb{K}(t^{10},u^{14})+\mathbb{K}\left(t^{11},u^{15}\right)+$$ $$+ (t^{12},u^{16}) \left(\mathbb{K}[[t]] \times \mathbb{K}[[u]] \right).$$
	Notice that the fact that we know the conductor of $R^{*}$ allows us to simplify some of the elements corresponding to the small elements by truncating the terms that have order greater than the conductor.
	
\end{ex}

\section{The algorithm in the general case}
\label{section3}
In this section we explain how to generalize the algorithm presented in the previous one to algebroid curve with more than two branches.
First of all we fix the notations.
We want to find the Arf closure of the ring $R \subseteq \mathbb{K}[[t_1]]\times \dots \times \mathbb{K}[[t_n]]$ with the following parametrization
$$ R=R_1=\mathbb{K}[[ \left( \phi_{11}(t_1),\ldots,\phi_{1n}(t_n)\right),\ldots, \left( \phi_{k1}(t_1),\ldots,\phi_{kn}(t_n)\right)]].$$
Similarly to the previous section, we will always replace an element of the parametrization $y=\left( \phi_{j1}(t_1),\ldots,\phi_{jn}(t_n)\right)$ such that $$ \textrm{ord}(\phi_{ji}(t_i))=0 \textrm{ and with } \phi_{j1}(0)=\phi_{ji}(0)\textrm{ for all } i=1,\ldots,n,    $$
with the element $\overline{y}=y-\phi_{j1}(0)\cdot (1,\ldots,1)$.

To compute the Arf closure $R^{*}$ we have to find the sequence of blow-ups $R_m$ of $R$. We will give an inductive algorithm for the computation of $R_m$.

We will denote by 
$$ R_m=\mathbb{K}\left[\left[ \left( \phi^{(m)}_{11}(t_1),\ldots,\phi^{(m)}_{1n}(t_n)\right),\ldots, \left( \phi^{(m)}_{k(m)1}(t_1),\ldots,\phi^{(m)}_{k(m)n}(t_n)\right)\right]\right].$$
If $i,j \in \left\{ 1,\ldots,n\right\}$ with $i\neq j $ we denote by $\pi_{i,j}$ the projection 
$$ \pi_{i,j}:  \mathbb{K}[[t_1]]\times \dots \times \mathbb{K}[[t_n]]  \to   \mathbb{K}[[t_i]]\times \mathbb{K}[[t_j]].$$
We have the following obvious Lemma:
\begin{lem}
	\label{L2}
	Consider $S \subseteq \mathbb{K}[[t_1]]\times \dots \times \mathbb{K}[[t_n]]$. We define the equivalence relation $\sim$ on $\left\{1,\ldots,n\right\}$, such that $i \sim j$ if $i=j$ or if $\pi_{i,j}(S)$ is local in $\mathbb{K}[[t_i]]\times \mathbb{K}[[t_j]]$.
	Then the partition $\mathfrak{P}(S)$,  defined in the first section, is the partition of $\left\{1,\ldots,n\right\}$ into equivalence classes with respect to $\sim$.
\end{lem}
If $$ S=\mathbb{K}[[ \left( \phi_{11}(t_1),\ldots,\phi_{1n}(t_n)\right),\ldots, \left( \phi_{k1}(t_1),\ldots,\phi_{kn}(t_n)\right)]],$$
then $$\pi_{i,j}(S)=\mathbb{K}[[ \left( \phi_{1i}(t_i),\phi_{1j}(t_j)\right),\ldots, \left( \phi_{ki}(t_i),\phi_{kj}(t_j)\right)]];$$
since in the two branches case we know how to understand if a ring is local from its parametrization, we have the following algorithm to compute $\mathfrak{P}(S)$:

\begin{algorithm}
	\SetKwData{Left}{left}
	\SetKwData{This}{this}
	\SetKwData{Up}{up}
	\SetKwFunction{Union}{Union}
	\SetKwFunction{FindCompress}{FindCompress}
	\SetKwInOut{Input}{input}
	\SetKwInOut{Output}{output}
	
	\caption{}
	\Input{$  S=\mathbb{K}[[ \left( \phi_{11}(t_1),\ldots,\phi_{1n}(t_n)\right),\ldots, \left( \phi_{k1}(t_1),\ldots,\phi_{kn}(t_n)\right)]]$}
	\Output{The partition $\mathfrak{P}(S)$}
	\BlankLine
	$ N \longleftarrow \left\{ 1,\ldots,n\right\}$

	\For{$ i \in N$}{$ P_i \longleftarrow \left\{i\right\}$ 
		
		\For{ $j \in N_{>i}$} { \If {$\pi_{i,j}(S)$ is local} { $ P_i \longleftarrow P_i \cup \left\{j \right\}$
				
				$N  \longleftarrow N \setminus \left\{ j\right\}$}}}

	\Return $\mathfrak{P}(S)=\left\{P_1,P_{i_2},\ldots,P_{i_t}\right\}$
	\label{algo_disjdecomp}
	
\end{algorithm}
\newpage
Once we know  that $\mathfrak{P}(S)=\left\{ P_1,\ldots,P_t\right\}$, with $$ P_i=\left\{ q_{i,1},\ldots, q_{i,k(i)} \right\},$$
we have $S=\prod_{i=1}^{t}{S(P_i)}$, where
$$ S(P_i)=\mathbb{K}\left[\left[ \left( \phi_{1q_{i,1}}(t_{q_{i,1}}),\ldots,\phi_{1q_{i,k(i)}}(t_{q_{i,k(i)}})\right),\ldots, \left( \phi_{kq_{i,1}}(t_{q_{i,1}}),\ldots,\phi_{kq_{i,k(i)}}(t_{q_{i,k(i)}})\right)\right]\right].$$
Now we can give an algorithm for computing the blow-up sequence of $R$. We will do it by working on induction on the number $n$ of branches.
We need to show a procedure to compute $R_{m+1}$ from $R_m$.

\begin{itemize}
	\item Base: $n=2$.
	
	For $n=2$ we have already seen, in the previous section, how to compute the $R_m$.
	
	\item Inductive step.
	
	We suppose that we are able to solve the problem for rings with less than $n$ branches and we give a procedure for rings with exactly $n$ branches.
	
	We have two cases:

	If $R_m$ is local we denote by $x_{R_m}$ an element of $R_m$ such that $\nu(x_{R_m})=\textrm{mult}(R_m)$ (we can find it as a linear combinations of the elements of the parametrization of $R_m$).
	
	Then we know that 
	
	$$ R_{m+1}=\mathbb{K}\left[\left[x_{R_m}, \frac{\left( \phi^{(m)}_{11}(t_1),\ldots,\phi^{(m)}_{1n}(t_n)\right)}{x_{R_m}},\ldots, \frac{\left( \phi^{(m)}_{k(m)1}(t_1),\ldots,\phi^{(m)}_{k(m)n}(t_n)\right)}{x_{R_m}}\right]\right].$$
	
	If $R_m$ is not local then we have that there exist a partition $\mathfrak{P}(R_m)=\left\{P_1,\ldots,P_t \right\}$ such that
	$$ R_m=\prod_{i=1}^{t} R_m(P_i). $$
	Notice that the $R_m(P_i)$ can be computed from the parametrization of $R_m$ and they are local rings with less then $n$ branches. Then for the inductive step we know how to compute the blow-up $\textrm{Bl}(R_m(P_i))$ of $R_m(P_i)$ and we have that:
	
	$$ R_{m+1}= \prod_{i=1}^{t} \textrm{Bl}(R_m(P_i)).$$
	
\end{itemize}

\begin{oss}
	\label{R5}
	It is clear that, with our definitions, we have $$S= \mathbb{K}[[t_1]]\times \dots \times \mathbb{K}[[t_n]] \iff \textrm{mult}^*(S)=\left\{ (1,0,\ldots,0), (0,1,0,\ldots,0),\ldots,(0,\ldots,0,1) \right\}.$$ \end{oss}
So we have a procedure to find the first $N$ such that $R_N=\mathbb{K}[[t_1]]\times \dots \times \mathbb{K}[[t_n]]$. From this procedure we can find the sequence 

$$ \textrm{mult}^{*}(R_1), \textrm{mult}^{*}(R_2), \ldots, \textrm{mult}^{*}(R_N),$$
from which we can build the multiplicity tree of $R^{*}$ up to level $N$. 
Once we know the multiplicity tree $T$ and the minimal tree we are able to give an expression for the Arf closure $R^{*}$ using the strategy presented in the previous section. In fact we can compute the conductor $c$ of the semigroup of values of the Arf closure and then  use the correspondence between the small elements of the Arf semigroup $\nu(R^*)$ and the elements of $R^{*}$ to find $\left\{ s_1,\ldots, s_l=c \right\} \subseteq R^{*}$ such that:
$$ R^{*}=\mathbb{K}(1,\ldots,1)+\mathbb{K}s_1+\ldots+\mathbb{K}s_{l-1}+(t_1^{c[1]},\ldots,t_n^{c[n]})  \left(\mathbb{K}[[t_1]]\times \dots \times \mathbb{K}[[t_n]]\right).$$

\begin{ex}
	\label{E2}
	
	We want to compute the Arf closure of the following ring 
	$$ R=R_1=\mathbb{K}[[(t^5-t^8,u^2+u^6,v^3,w^2+w^9),(t^{6},u^2+u^7+u^{10},v^{7}-v^9,w^2+w^7)]].$$
	In order to simplify the notation we will set $R_i^j=R_i(P_j)$, $x_i^j=x_{R_i}^j$ and we denote by $x_{R_i^j}$ the element of minimal valuation  in the local ring $R_i^j$.
	
	It is easy to verify that $\pi_{1,2}(R),\pi_{1,3}(R)$ and $\pi_{1,4}(R)$ are all  local. Then for the Lemma \ref{L2} follows that $\mathfrak{P}(R)=\left\{ \left\{1,2,3,4\right\} \right\}$, therefore $R$ is local.
	
	We have that $\textrm{mult}(R_1)=(5,2,3,2)$. As the minimal element $x_{R_1}$ we can choose $x_{R_1}=(t^5-t^8,u^2+u^6,v^3,w^2+w^9)$.
	
	We have:
	$$ R_2=\mathbb{K}\left[ \left[x_{R_1}=(t^5-t^8,u^2+u^6,v^3,w^2+w^9),\frac{(t^{6},u^2+u^7+u^{10},v^{7}-v^9,w^2+w^7)}{x_{R_1}} \right]\right]=$$ $$=
	\mathbb{K}\left[\left[ (t^5-t^8,u^2+u^6,v^3,w^2+w^9), \left( \frac{t}{1-t^3}, \frac{1+u^5+u^8}{1+u^4},v^4-v^6,\frac{1+w^5}{1+w^7}\right) \right]\right]. $$
	
	Now we can verify that $\pi_{1,2}(R_2)$ is not local, $\pi_{1,3}(R_2)$ is local, $\pi_{1,4}(R_2)$ is not local and $\pi_{2,4}(R_2)$ is local, therefore $\mathfrak{P}(R_2)=\left\{P_1=\left\{1,3\right\}, P_2=\left\{2,4 \right\}\right\}$.
	We have 
	$$ R_2=R_2^1 \times R_2^2,$$
	where $$ R_2^1=\mathbb{K}\left[\left[ (t^5-t^8,v^3), \left( \frac{t}{1-t^3},v^4-v^6 \right) \right]\right],$$
	$$ R_2^2=\mathbb{K}\left[\left[ (u^2+u^6,w^2+w^9),\left(\frac{1+u^5+u^8}{1+u^4},\frac{1+w^5}{1+w^7}\right)\right]\right]=$$ $$=\mathbb{K}\left[\left[ (u^2+u^6,w^2+w^9),\left(\frac{-u^4+u^5+u^8}{1+u^4},\frac{w^5-w^7}{1+w^7}\right)\right]\right]$$
	where, following our conventions on the parametrization, we  replace $\left(\frac{1+u^5+u^8}{1+u^4},\frac{1+w^5}{1+w^7}\right)$ with  $\left(\frac{1+u^5+u^8}{1+u^4},\frac{1+w^5}{1+w^7}\right)-(1,1)=\left(\frac{-u^4+u^5+u^8}{1+u^4},\frac{w^5-w^7}{1+w^7}\right)$.

	We have $\textrm{mult}(R_2^1)=(1,3)$ and we can choose as element of minimal value the sum $x_{R_2^1}$ of its two generators $$x_{R_2^1}=\left( \frac{t+t^5(1-t^3)^2}{1-t^3},v^3+v^4-v^6\right),$$ while $\textrm{mult}(R_2^2)=(2,2)$ and we can choose as element of minimal value \\ $ \displaystyle x_{R_2^2}= \left(u^2+u^6,w^2+w^9\right)$.
	Then we have $\textrm{mult}^{*}(R_2)=\left\{ (1,0,3,0), (0,2,0,2) \right\}$ and we can proceed with the computation of $R_3$.
	Thus
	$$R_3=\textrm{Bl}(R_2^1) \times \textrm{Bl}(R_2^2) ,$$
	so we have to compute $\textrm{Bl}(R_2^1)$ and $\textrm{Bl}(R_2^2).$

	We have
	$$\textrm{Bl}(R_2^1)=\mathbb{K}\left[\left[\left(\phi_1^{(3)}(t),\psi_1^{(3)}(v)\right),\ldots, \left(\phi_{3}^{(3)}(t),\psi_{3}^{(3)}(v)\right)\right]\right],$$
	
	where
	\begin{itemize}
		
		\item $ \displaystyle \left(\phi_1^{(3)}(t),\psi_1^{(3)}(v)\right)=\left( \frac{t+t^5(1-t^3)^2}{1-t^3},v^3+v^4-v^6\right);$
		
		\item  $ \displaystyle \left(\phi_2^{(3)}(t),\psi_2^{(3)}(v)\right)=\left(\frac{t^4(1-t^3)^2}{1+t^4(1-t^3)^2},\frac{1}{1+v-v^3}\right);$
		\item  $ \displaystyle \left(\phi_{3}^{(3)}(t),\psi_{3}^{(3)}(v)\right)=\left(\frac{1}{1+t^4(1-t^3)^2},\frac{v-v^3}{1+v-v^3}\right).$
	\end{itemize}
	We notice that the second generator has valuation $(4,0)$, then $\textrm{Bl}(R_2^1)$ is not local in $\mathbb{K}[[t]] \times \mathbb{K}[[v]]$. Furthermore we have, with our notation, that $\textrm{mult}^*(\textrm{Bl}(R_2^1))=\left\{(1,0),(0,1)\right\}$.
	Then we have 
	
	$$ \textrm{Bl}(R_2^1)= \mathbb{K}[[t]] \times \mathbb{K}[[v]].$$
	Now we can compute $\textrm{Bl}(R_2^2) $. We have
	$$\textrm{Bl}(R_2^2)=\mathbb{K}\left[\left[\left(u^2+u^6,w^2+w^9\right),\left( \frac{-u^2+u^3+u^6}{(1+u^4)^2},\frac{w^3-w^5}{(1+w^7)^2}\right)\right]\right].$$
	Then we have that $\textrm{Bl}(R_2^2)$ is local in $ \mathbb{K}[[u]] \times \mathbb{K}[w]]$ and $\textrm{mult}(\textrm{Bl}(R_2^2))=(2,2)$. 
	Then $\mathfrak{P}(R_3)=\left\{ P_1=\left\{1\right\}, P_2=\left\{3\right\}, P_3=\left\{2,4 \right\}\right\} $ and
	$$ R_3=R_3^1\times R_3^2\times R_3^3=\mathbb{K}[[t]]\times\mathbb{K}[[v]] \times \textrm{Bl}(R_2^2),$$
	with $\textrm{mult}^{*}(R_3)=\left\{ (1,0,0,0), (0,0,1,0) ,(0,2,0,2)\right\}$.
	As a minimal element of $R_3^3$ we can choose again $ x_{R_3^3}=\left(u^2+u^6,w^2+w^9\right)$.
	Thus $$R_4=\textrm{Bl}(\mathbb{K}[[t]]) \times \textrm{Bl}(\mathbb{K}[[v]]) \times \textrm{Bl}(R_3^3)=\mathbb{K}[[t]] \times \mathbb{K}[[v]] \times \textrm{Bl}(R_3^3).$$
	We have:
	$$ \textrm{Bl}(R_3^3)=\mathbb{K}\left[\left[\left(u^2+u^6,w^2+w^9\right),\left( \frac{-1+u+u^4}{(1+u^4)^3},\frac{w-w^3}{(1+w^7)^3}\right)\right]\right].$$
	From this it is easy to show that $\textrm{Bl}(R_3^3)= \mathbb{K}[[u]] \times  \mathbb{K}[[w]]$. 
	
\noindent	Then $\mathfrak{P}(R_4)=\left\{ P_1=\left\{1\right\}, P_2=\left\{2\right\}, P_3=\left\{3 \right\},P_4=\left\{4 \right\}\right\} $ and
	$$R_4=\mathbb{K}[[t]] \times  \mathbb{K}[[u]] \times \mathbb{K}[[v]] \times  \mathbb{K}[[w]],$$
	and we have reached the stop condition for our algorithm.
	We found that $N=4$ and
	\begin{itemize}
		\item $ \textrm{mult}^{*}(R_1)=\left\{ (5,2,3,2) \right\},$  \item $\textrm{mult}^{*}(R_2)=\left\{ (1,0,3,0), (0,2,0,2) \right\}, $ \item $\textrm{mult}^{*}(R_3)=\left\{ (1,0,0,0), (0,0,1,0) ,(0,2,0,2)\right\}, $ \item $\textrm{mult}^{*}(R_4)=\left\{ (1,0,0,0), (0,1,0,0) ,(0,0,1,0),(0,0,0,1)\right\}.  $
		
	\end{itemize}
	The corresponding minimal elements are:
	
	\begin{itemize}
		\item $x_{R_1}=(t^5-t^8,u^2+u^6,v^3,w^2+w^9),$  \item $x_2^{1}=\left( \frac{t+t^5(1-t^3)^2}{1-t^3},1,v^3+v^4-v^6,1\right)$ and $x_2^2= \left(1,u^2+u^6,1,w^2+w^9\right);$\item $ x_3^1=(t,1,1,1), x_3^2=(1,1,v,1)$ and $ x_3^3= \left(1,u^2+u^6,1,w^2+w^9\right);$ \item $x_4^1=(t,1,1,1), x_4^2=(1,u,1,1), x_4^3=(1,1,v,1)$ and $x_4^4=(1,1,1,w)$.
		
	\end{itemize}
	Then we have the following trees:
	
	
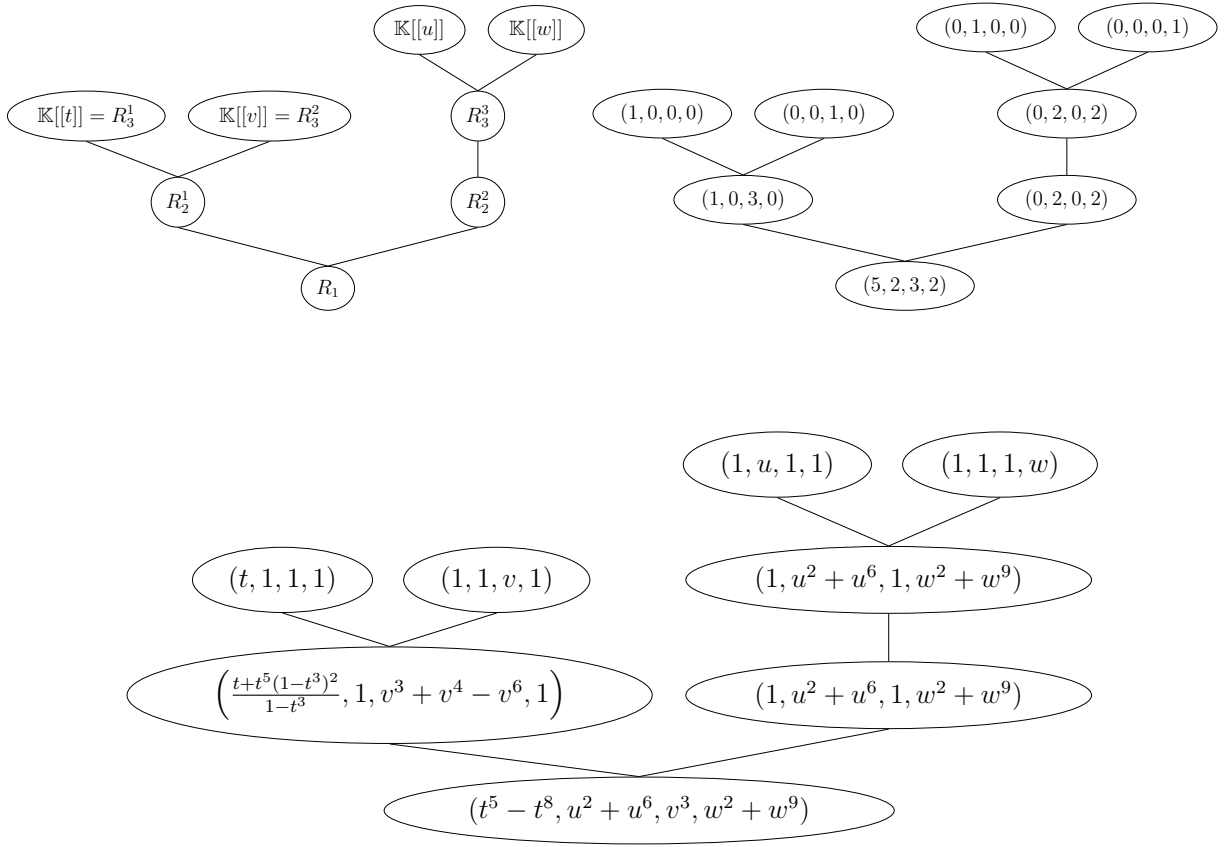
\begin{center}
		\begin{tikzpicture}[grow'=up,sibling distance=13pt,scale=.65]
		\tikzset{level distance=50pt,every tree node/.style={draw,ellipse}} \Tree [ .$R_1$ [ .$R_2^1$ [ .$\K[[t]]=R_3^1$ ] [ .$\K[[v]]=R_3^2$         
		]  ] [.$R_2^2$ [ .$R_3^3$ [ .$\K[[u]]$ ] [ .$\K[[w]]$ ] ] ] ]  \end{tikzpicture} 
		\begin{tikzpicture}[grow'=up,sibling distance=13pt,scale=.65]
		\tikzset{level distance=50pt,every tree node/.style={draw,ellipse}} \Tree [ .$(5,2,3,2)$ [ .$(1,0,3,0)$ [ .$(1,0,0,0)$ ] [ .$(0,0,1,0)$         
		]  ] [.$(0,2,0,2)$ [ .$(0,2,0,2)$ [ .$(0,1,0,0)$ ] [ .$(0,0,0,1)$ ] ] ] ]  \end{tikzpicture} \end{center}
	\vspace{1cm}
	\begin{center}
		\begin{tikzpicture}[grow'=up,sibling distance=13pt,scale=.87]
		\tikzset{level distance=50pt,every tree node/.style={draw,ellipse}} \Tree [ .$(t^5-t^8,u^2+u^6,v^3,w^2+w^9)$ [ .{$\left(\frac{t+t^5(1-t^3)^2}{1-t^3},1,v^3+v^4-v^6,1\right)$} [ .$(t,1,1,1)$ ] [ .$(1,1,v,1)$         
		]  ] [.$(1,u^2+u^6,1,w^2+w^9)$ [ .$(1,u^2+u^6,1,w^2+w^9)$ [ .$(1,u,1,1)$ ] [ .$(1,1,1,w)$ ] ] ] ]  \end{tikzpicture} \end{center}
	Then the multiplicity tree $T(R)$ of the Arf semigroup associated to $R^{*}$ is the tree described by the matrix $$ M(T(R))_E= \left(  \begin{matrix} 0 &1& 2& 1 \\ 0& 0& 1& 3 \\ 0& 0 &0 &1 \\ 0& 0 &0 &0 \\\end{matrix}\right),$$
	where $E=\left\{ [5],[2,2,2],[3,3],[2,2,2]\right\}$.
	
	The conductor of $\nu(R^*)$ is $c=(6,6,6,6)$, therefore $$(t^{6},u^6,v^{6},w^6) \left(\mathbb{K}[[t]] \times  \mathbb{K}[[u]] \times \mathbb{K}[[v]] \times  \mathbb{K}[[w]]\right) \subseteq R^{*}.$$
	We have that
	$$\textrm{Small}(\nu(R^*))=\left\{ ( 5, 2, 3, 2), ( 5, 4, 3,4 ), ( 5, 6, 3, 6 ), ( 6, 2,6, 2), ( 6, 4,6, 4 ), c=( 6, 6,6, 6 ) \right\}.$$
	From the minimal tree we can recover the elements of $R^{*}$ with valuation belonging to $\textrm{Small}(\nu(R^*))$. We can calculate the Arf closure truncating the terms with degree bigger than the conductor. So we obtain:
	$$\left\{ (t^5,u^2,v^3,w^2), (t^ 5, u^4, v^3, w^4 ), ( t^5, 0, v^3, 0 ), ( 0, u^2, 0, w^2 ), ( 0, u^4, 0, w^4 ) \right\}.$$
	Finally we have 
	$$ R^{*}=\mathbb{K}(1,1,1,1)+\mathbb{K} (t^5,u^2,v^3,w^2)+\mathbb{K}(t^ 5, u^4, v^3, w^4 )+ \mathbb{K}( t^5, 0, v^3, 0 )+$$$$+\mathbb{K}( 0, u^2, 0, w^2 )+\mathbb{K}  ( 0, u^4, 0, w^4 ) +(t^{6},u^6,v^6,w^6) \left(\mathbb{K}[[t]] \times  \mathbb{K}[[u]] \times \mathbb{K}[[v]] \times  \mathbb{K}[[w]]\right).$$ \label{f}
\end{ex}

\section{A bound for the series}
\label{section4}
In the previous sections, we have presented an algorithm to compute the Arf closure of an algebroid curve. Now, we would like to find a bound for the truncation of the series expansion in the parametrization, in order to improve the speed of the algorithm.\\
Our strategy is based on the following theorem that generalizes the Arslan-Sahin theorem to the case of two branches algebroid curves. Thus, in the following, we focus on the two branches case.\\
Let us fix some notation.
Let $R$ be a two-branches curve with parametrization
$$R=\K[[(\phi_1(t),\psi_1(u)),\ldots,(\phi_n(t),\psi_n(u))]],$$
 we call $c=(c[1],c[2])$ the conductor of $\nu(R^*)$. 
Furthermore, we denote by $\overline{\phi_i(t)}$ and $\overline{\psi_i(u)}$  the formal power series obtained from $\phi_i(t)$ and $\psi_i(u)$ respectively by removing all elements with order greater than $c[1]+1$ and $c[2]+1$.
 Finally, we introduce: $$\overline{R}=\K[[(\overline{\phi_1(t)},\overline{\psi_1(u)}),\ldots,(\overline{\phi_n(t)},\overline{\psi_n(u)})]].$$

\begin{teo} 
	\label{T1}
	If we apply the algorithm to both $R$ and $\bar{R}$ we obtain the same multiplicity tree.
\end{teo}
\begin{proof}
	
	Let us start writing the representation of an arbitrary element of the parametrization of $R$,
	$$(\phi_i^{(1)}(t),\psi_i^{(1)}(u))=\left(\sum_{i\leq c[1]+1}a_it^i+\sum_{i>c[1]+1}a_it^i,\sum_{i\leq c[2]+1}b_iu^i+\sum_{i>c[2]+1}b_iu^i\right).$$
	We denote by
	$$(\chi_1^{(1)}(t),\chi_2^{(1)}(u))=\left(\sum_{i>c[1]+1}a_it^i,\sum_{i>c[2]+1}b_iu^i\right)$$
	and
	$$k=(k[1],k[2])=(ord(\chi_1^{(1)}(t)), ord(\chi_2^{(1)}(u)))>(c[1]+1,c[2]+1).$$
	Now, we want to follow the path of $\chi_1^{(1)}(t)$ and $\chi_2^{(1)}(u)$ in the algorithm in order to observe that by removing them from parametrization, the result of the algorithm remains unchanged. We denote with $(\chi_1^{(i)}(t),\chi_2^{(i)}(u))$ the series obtained by $(\chi_1^{(1)}(t),\chi_2^{(1)}(u))$ at $i$-th step of the algorithm.
	
	To prove the thesis, it is necessary to
	prove that $(\chi_1^{(i)}(t),\chi_2^{(i)}(u))$ satisfy the following hypothesis at the $i$-th step:
	\begin{itemize}
		\item[i)] $ord(\chi_1^{(i)}(t))> M_1[i]$ and $ord(\chi_2^{(i)}(u))> M_2[i]$;
		\item[ii)] neither $ord(\chi_1^{(i)}(t))$ nor $ord(\chi_2^{(i)}(u))$ are $0$.
	\end{itemize}
	If $i)$ is true we have that the monomials in $(\chi_1^{(i)}(t),\chi_2^{(i)}(u))$ are not involved in the choice of the minimal valuation elements at the $i$-th step. If $ii)$ is true they are not involved in the splits as consequence of Lemma \ref{L1}.\\
	So, if both hypothesis are true, the monomials in $(\chi_1^{(i)}(t),\chi_2^{(i)}(u))$ are not involved in the $i$-th step of the algorithm.\\
	If $p_1$ is the highest level were the branches in $R$ are joined, for all $i\leq p_1$, we have that:
	{\small
		\begin{eqnarray*}
			\nu(\chi_1^{(i)}(t),\chi_2^{(i)}(u))& \geq &(k[1]-M_1[1]-\ldots-M_1[i-1],k[2]-M_2[1]-\ldots-M_2[i-1])>\\
			&>&(c[1]+1-M_1[1]-\ldots-M_1[i-1],c[2]+1-M_2[1]-\ldots-M_2[i-1])=\\
			&=& \left(\sum_{j=1}^{\max(l_1,p_1)}{M_1[j]}+1- \sum_{j=1}^{i-1}{M_1[j]},\sum_{j=1}^{\max(l_2,p_1)}{M_2[j]}+1- \sum_{j=1}^{i-1}{M_2[j]} \right) \geq \\
			&\geq& \left(\sum_{j=1}^{i}{M_1[j]}+1- \sum_{j=1}^{i-1}{M_1[j]},\sum_{j=1}^{i}{M_2[j]}+1- \sum_{j=1}^{i-1}{M_2[j]} \right) =\\
			&=&(M_1[i]+1,M_2[i]+1) >(M_1[i],M_2[i])>(0,0)	\end{eqnarray*}}
\noindent	So the hypothesis $i)$ and $ii)$ are satisfied for $\chi_1^{(i)}(t),\chi_2^{(i)}(u)$ with $i\leq p_1$. When $i > p_1$ the algorithm works individually on each branch, therefore is the same as the one presented by Arslan-Sahin. Thus, because  we have that $\chi_1^{(p_1+1)}(t)$ and $\chi_2^{(p_1+1)}(u)$ are element with valuation strictly greater then the conductor of $R_1^{(p_1+1)}$ and $R_2^{(p_1+1)}$ respectively,  for the Arslan-Sahin theorem (cf.\cite[Thm. 2.4]{Arf:closure}), $\chi_1^{(p_1+1)}(t),\chi_2^{(p_1+1)}(u)$ are not involved in the next steps of the algorithm and this concludes the proof.
\end{proof}

\begin{oss}
	\label{R6}
	We want to point out that the previous theorem does not imply that the chains of blow-ups obtained applying the algorithm on $R$ and $\bar{R}$ are the same. In general, the parametrization of each blow-up and the minimal tree are different, but they are equal  modulo $\langle t^{c+2}, u^{c+2} \rangle $ (when we truncate all the elements of degree greater than $c+1$).
\end{oss}

In the previous section, we have computed a presentation of the Arf closure starting by any minimal tree of the curve and it does not depends on the minimal tree chosen. For this reason we can enunciate the following obvious corollary.

\begin{cor}
	\label{C1}
	Using the same notation of previous theorem. $R$ and $ \overline{R} $ have the same Arf closure.
\end{cor}

From the previous Corollary it follows that our new problem is to find  a way to estimate the conductor of $\nu(R^*)$ without actually knowing $R^*$. Now we see how to do that by using the information given by the starting parametrization of $R$.
Let us start by considering separately the two branches:
$$R^1=\K[[\phi_1(t),\ldots, \phi_n(t)]]\hspace{1cm}R^2=\K[[\psi_1(u),\ldots, \psi_n(u)]].$$
As we saw in the Remark \ref{R4} , it is possible to apply the Arslan-Sahin's algorithm in order to find the multiplicity sequences $M_1$ and $M_2$ of the two branches.
$M_1$ and $M_2$ are multiplicity sequences so they must satisfy the following property:$$ \forall k\geq1 \textrm{ there exist }  s_{1,k} \textrm{ and } s_{2,k}\in \mathbb{N}, \textrm{ such that }  s_{i,1}\geq k+1, s_{i,2}\geq k+1 \textrm{ and }$$
$$M_1[k]=\sum_{j=k+1}^{s_{1,k}}M_1[j],\hspace{1cm}M_2[k]=\sum_{j=k+1}^{s_{2,k}}M_2[j].$$
If $L=\max\{l_1,l_2\}$ we can define the following vectors:
$$S(1)=[s_{1,1},s_{1,2},\ldots,s_{1,L}],\hspace{1cm} S(2)=[s_{2,1},s_{2,2},\ldots,s_{2,L}].$$
Now let us consider the set $D(1,2)=\{k:s_{1,k}\neq s_{2,k}\}$ and we  suppose that $D(1,2)\neq \emptyset$ (i.e the two sequences are not equal). In this case we define $k_E(1,2)=\min\{\min(s_{1,k},s_{2,k}): k\in D(1,2)\}$. We present the following theorem, that was proved in (\cite{G:Z})

\begin{teo}\cite[Prop 1.2]{G:Z}
	\label{T2}
	If $T$ is the tree of an algebroid two-branches curve with $D(1,2)\neq \emptyset$, then $k_E(1,2)+1$ is the lowest level where the two branches are prevented from being glued in $T$; in other words $p_1\leq k_E(1,2)$.
\end{teo}
\begin{proof}
	Suppose by contradiction that the first and the second branches are glued at level $k_E(1,2)+1$. From the definition of $k_E(1,2)+1$, there exists $\overline{k}\in D(1,2)$ such that $k_E(1,2)=\min\{s_{1,\overline{k}}, s_{2,\overline{k}}\}$. Without loss of generality suppose that $\min\{s_{1,\overline{k}}, s_{2,\overline{k}}\}=s_{1,\overline{k}}$ (where $s_{1,\overline{k}}\neq s_{2,\overline{k}}$).
	
	We have the following nodes in the multiplicity tree
	$$(M_1[\overline{k}],M_2[\overline{k}]),\ldots,(M_1[k_E(1,2)],M_2[k_E(1,2)]),\ldots,(M_1[k_E(1,2)+1],M_2[k_E(1,2)+1]),\ldots$$
	We have that $k_E(1,2) = s_{1,\overline{k}}$ so
	$$M_1[\overline{k}]=\sum_{i=\overline{k}+1}^{k_E(1,2)}M_1[i]$$
	while $k_E(1,2)+1 = s_{1,\overline{k}}+1\leq s_{2,\overline{k}}$ so
	$$M_2[\overline{k}]=\sum_{i=\overline{k}+1}^{s_{2,\overline{k}}}M_2[i]\geq \sum_{i=\overline{k}+1}^{k_E(1,2)+1}M_2[i]$$
	These facts easily imply that the node $(M_1[\overline{k}],M_2[\overline{k}])$ cannot be expressed as a sum of the nodes of a
	subtree rooted in it, so we have a contradiction. Two branches are forced to split up only when
	we have this kind of problem, so the minimality of $k_E(1,2)$ guarantees that they can be glued at
	level $k_E(1,2)$ (and obviously at lower levels).
\end{proof}

If we set:
$$d_1=\max\{l_1,k_E(1,2)\},\hspace{1cm} d_2=\max\{l_2,k_E(1,2)\},$$
we have:
\begin{eqnarray*}
	&c[1]+1= \sum_{i=1}^{\max(l_1,p_1)}{M_1[i]}+1 \leq M_1[1]+\ldots+M_1[d_1]+1,\\
	&c[2]+1=\sum_{i=1}^{\max(l_2,p_1)}{M_2[i]}+1 \leq M_2[1]+\ldots+M_2[d_2]+1.
\end{eqnarray*}

So, if we put:
\begin{eqnarray*}
	&b_1=M_1[1]+\ldots+M_1[d_1]+1,\\
	&b_2=M_2[1]+\ldots+M_2[d_2]+1,
\end{eqnarray*}
as consequence of the Theorem \ref{T1}, we can use the vector $b_O=(b_1,b_2)$ as a bound for the series expansions in the parametrizations.\\
We have found a bound when $D(1,2)\neq \emptyset$ by only using the numeric properties of the multiplicity sequences. When $D(1,2)= \emptyset$ we cannot make assumptions on the split level by only using the $M_i$ but we need to work directly on the parametrization in order to find a suitable bound.

Let us suppose that we have an algebroid curve with two branches and $D(1,2)=\emptyset$.
In this case we will do the following positions in order to simplify the notation.
We denote with $c_r$ the conductor of the branches $R^{1}$ and $R^{2}$ (in fact, in this case the two conductors are equal). We  also set $l=l_1=l_2$.
Now we define $Dis(1,2)=\{i\in \{1,\ldots,n\}: \nu(\phi_i(t))\neq \nu(\psi_i(u))\}$ and we call discrepancies the elements of this set.
If $Dis(1,2) \neq \emptyset $, we define also $$D=\min\{\min\{\nu(\phi_i(t)),\nu(\psi_i(u))\},\hspace{0.1cm}i\in Dis(1,2)\}$$
which is the smallest order that causes a discrepancy.
\begin{ex}
\label{E3}
	Let us consider the algebroid curve:
	$$R=\K[[(t^3+t^4,u^3+u^7),(t^8+t^9,u^8),(t^{12}+t^{15},u^{13}+u^{14}),(t^{21},u^{17}+u^{19}
	)]].$$
	The multiplicity tree associated to the ring is:
	
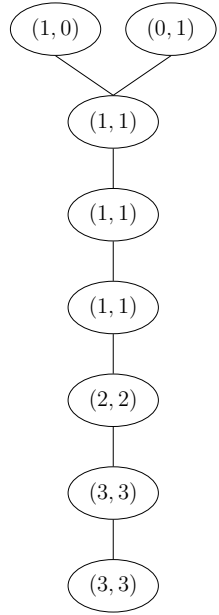
\begin{center}
		\begin{tikzpicture}[grow'=up,sibling distance=13pt,scale=.70]
		\tikzset{level distance=50pt,every tree node/.style={draw,ellipse}} \Tree [ .$(3,3)$ [ .$(3,3)$ [ .$(2,2)$ [ .$(1,1)$  [.$(1,1)$ [ .$(1,1)$ [ .$(1,0)$ ] [ .$(0,1)$ ] ] ] ]]]]  \end{tikzpicture} \end{center}
	So we have: $D(1,2)=\emptyset$, $Dis(1,2)=\{3,4\}$ and $$D=\min\{\min\{12,13\},\min\{21,17\}\}=\min\{12,17\}=12.$$
\end{ex}
\begin{lem}
	\label{L3}
	Let
	$$R=\K[[(\phi_1(t),\psi_1(u)),\ldots,(\phi_n(t),\psi_n(u))]]$$
	be an algebroid branch such that
	\begin{itemize}
		\item[i)] $M_1=M_2$ ($D(1,2)=\emptyset$);
		\item[ii)] $Dis(1,2)\neq \emptyset$.
	\end{itemize}
Then we have $\max\{c_r,D\}\geq c[1]=c[2]$.
\end{lem}
\begin{proof}
From the definition of $D$, it follows that there exists an element of the type $(D,x)$ in $\nu(R) \subseteq \nu(R^*)$ with $x>D$ (or equivalently of the type $(y,D)$ with $y>D$).
We know that there exists an integer $k$ such that $$D=\sum_{i=1}^k{M_1[i]}.$$
Taking in account that the multiplicity tree $T(R)$ has two identical branches, it is easy to understand that $(D,x) \in \nu(R^*)$ with $x>D$ implies  $p_1 \leq k$ (if we had $k<p_1$ the only possible element with valuation of the type $(D,x)$ in $\nu(R^*)$ would be $(D,D)$).
So we have 

$$ c[2]=c[1]=\sum_{i=1}^{\max(l_1,p_1)}{M_1[i]}\leq \sum_{i=1}^{\max(l_1,k)}{M_1[i]}=\max\{c_r,D\}. $$
\end{proof}

As a consequence of this theorem, we can take $b_D=(\max\{c_r,D\}+1,\max\{c_r,D\}+1)$ as a bound for an algebroid curve with $D(1,2)=\emptyset$ and $Dis(1,2)\neq \emptyset$.

Now we only need to understand how to deal with the case of algebroid curves with $D(1,2)=\emptyset$ and $Dis(1,2)=\emptyset$.
In this case we have:
\begin{itemize}
	\item[i)] $M_1=M_2$;
	\item[ii)] $\nu(\phi_i(t))=\nu(\psi_i(u))$ $\forall i=1,\ldots, n$.
\end{itemize}

Without loss of generality, we can rename the elements of the parametrization  in order to have:
$$\nu(\phi_1(t),\psi_1(u))\leq \nu(\phi_2(t),\psi_2(u))\leq \ldots \leq \nu(\phi_n(t),\psi_n(u)).$$
Let $(\phi_i(t),\psi_i(u))$ be the first element with $i>1$ such that at least one of the following holds
\begin{itemize}
	\item $\phi_1(t)\neq \psi_1(t)$
	\item $\phi_i(t)\neq \psi_i(t)$,
\end{itemize}  (it  must exist an element of this type because  otherwise we would not have an algebroid curve).
In this case we can always find $a,b,r,s\in \N$, such that
$$({\tilde{\phi}(t)},{\tilde{\psi}(u)})=a(\phi_1(t),\psi_1(u))^r+b(\phi_i(t),\psi_i(u))^s$$
with $ord(\tilde{\phi}(t)) > ord(\phi_1(t))$.\\

Now let us consider
$$\tilde{R}=\K[[(\tilde{\phi}(t),\tilde{\psi}(u)),(\phi_2(t),\psi_2(u)),\ldots, (\phi_n(t),\psi_n(u))]]$$
and denote with $\tilde{c}$ the conductor of the $\tilde{R}$ Arf closure.\\
Since $\tilde{R}\subseteq R$, we have $c\leq \tilde{c}$. Now, if $\tilde{R}$ is an algebroid curve where both $D(1,2)$ and $Dis(1,2)$ are not empty we have showed how to compute a bound for $\tilde{R}$ and this is also a bound for $R$ since $c\leq \tilde{c}$.\\ 
On the contrary, we can apply the same idea starting by $\tilde{R}$ until we found an algebroid curve with a discrepancy for which we know to compute a bound; we will call this bound $b_G$. 
We note that this process necessarily produces a discrepancy since R is an algebroid curve.
\begin{oss}
	\label{R7}
	We observe that it makes sense compute $b_G$ even when we have a discrepancy. A priori we do not  know in this case which bound is better between $b_D$ and $b_G$, so we will compute both of them and then we will choose the smaller one.
\end{oss}

We will enunciate the following proposition that summarizes what we have seen above.
\begin{prop}
\label{P1}

If $R$ is an algebroid curve and $c$ is the conductor of its Arf closure and we consider the element
$$b=\begin{cases}
b_O\hspace{5cm}&\text{if}\hspace{0.5cm} D(1,2)\neq \emptyset;\\
\min\{b_D,b_G\}\hspace{5cm}&\text{if}\hspace{0.5cm} D(1,2)=\emptyset\wedge Disc(1,2)\neq \emptyset;\\
b_G\hspace{5cm}&\text{if}\hspace{0.5cm} D(1,2)=\emptyset\wedge Disc(1,2)= \emptyset,\\
\end{cases}$$
we have $b\geq (c[1]+1,c[2]+1)$.

\end{prop}
As a consequence of the last proposition and Theorem \ref{T1}, we have that $b$ is a suitable bound for the algorithm.\\
Finally we  show how the bound found in two-branches case can be used to determine a bound in the general case.

\begin{oss}
\label{R8}
If $R$ is an algebroid curve with $n$ branches, parametrized by
$$ R=\mathbb{K}[[ \left( \phi_{11}(t_1),\ldots,\phi_{1n}(t_n)\right),\ldots, \left( \phi_{k1}(t_1),\ldots,\phi_{kn}(t_n)\right)]].$$
We consider
$$\pi_{i,j}(R)=\mathbb{K}[[ \left( \phi_{1i}(t_i),\phi_{1j}(t_j)\right),\ldots, \left( \phi_{ki}(t_i),\phi_{kj}(t_j)\right)]],$$
 the two-branch curve associated with the branches $i$ and $j$ for $i,j=1,\ldots,n$, $i\neq j$. We call $b_{\pi_{ij}(R)}=(b_{\pi_{ij}(R),i},b_{\pi_{ij}(R),j})$ the bound computed for the curve $\pi_{ij}(R)$ where $b_{\pi_{ij}(R),i}$ and $b_{\pi_{ij}(R),j}$ are the components of the bound related to the branches $i$ and $j$ respectively.
If we consider $$b[i]=\max\{b_{\pi_{i,j}(R),i} \hspace{0.5cm} j=1,\ldots,n,j\neq i\},$$
it is easy to observe that $b=(b[1],b[2],\ldots, b[n])$ is a suitable bound for the curve (because the general algorithm performs simultaneously the two case one on each couple of branches).
\end{oss}
\begin{ex}
	\label{E4}
	We want to compute, using the truncation explained in the previous section, the Arf closure of the  ring 
	$$ R=R_1=\mathbb{K}[[(t^5-t^8,u^2+u^6,v^3,w^2+w^9),(t^{6},u^2+u^7+u^{10},v^{7}-v^9,w^2+w^7)]],$$	
	that appeared in the Example \ref{E2}.
	
	If we use tha algorithm of Arslan and Sahin to compute the Arf closure of the rings
	$$ R^1=\mathbb{K}[[t^5-t^8,t^{6}]], R^2=\mathbb{K}[[u^2+u^6,u^2+u^7+u^{10}]], $$ $$ R^3=\mathbb{K}[[v^3,v^{7}-v^9]] , R^4=\mathbb{K}[[w^2+w^9,w^2+w^7]], $$
	we find that the multiplicity tree $T$ of $R^*$ belongs to $\tau(E)$, where $$ E=\left\{ M_1=[5], M_2=[2,2,2], M_3=[3,3], M_4=[2,2,2]\right\},$$
	where with $\tau(E)$ we indicate the family of all the multiplcity trees having multiplicity branches in $E$.
	
	We want compute the bounds $b_{\pi_{ij}(R),i}$ with $i,j=1,2,3,4$, $i\neq j$. Since $b_{\pi_{ij}(R),i}=b_{\pi_{ji}(R),i}$ for all $i,j=1,2,3,4$, $i\neq j$, we can reduce to compute only $b_{\pi_{ij}(R),i}$ where $j>i$.
	
	If $k_E(i,j) \neq \infty $ we have seen that:
	
	$$b_{\pi_{ij}(R),i}=\left(\sum_{k=1}^{\max(l_i,k_E(i,j))}{M_i[k]}\right)+1  \textrm{ and } b_{\pi_{ij}(R),j}=\left(\sum_{k=1}^{\max(l_j,k_E(i,j))}{M_j[k]}\right)+1.$$
	
	We have:
	\begin{itemize}
		\item $ k_E(1,2)=2 \Rightarrow$
		\begin{eqnarray*} b_{\pi_{12}(R),1}&=&\left(\sum_{k=1}^{\max(1,2)=2}{M_1[k]}\right)+1=5+1+1=7; \\ b_{\pi_{12}(R),2}&=&\left(\sum_{k=1}^{\max(3,2)=3}{M_2[k]}\right)+1=2+2+2+1=7. \end{eqnarray*}
		
		\item $ k_E(1,3)=2 \Rightarrow$
		\begin{eqnarray*} b_{\pi_{13}(R),1}&=&\left(\sum_{k=1}^{\max(1,2)=2}{M_1[k]}\right)+1=5+1+1=7; \\ b_{\pi_{13}(R),3}&=&\left(\sum_{k=1}^{\max(2,2)=2}{M_3[k]}\right)+1=3+3+1=7. \end{eqnarray*}
		\item $ k_E(1,4)=2 \Rightarrow$
		\begin{eqnarray*} b_{\pi_{14}(R),1}&=&\left(\sum_{k=1}^{\max(1,2)=2}{M_1[k]}\right)+1=5+1+1=7; \\ b_{\pi_{14}(R),4}&=&\left(\sum_{k=1}^{\max(3,2)=3}{M_4[k]}\right)+1=2+2+2+1=7. \end{eqnarray*}
		\item $ k_E(2,3)=3 \Rightarrow$
		\begin{eqnarray*} b_{\pi_{23}(R),2}&=&\left(\sum_{k=1}^{\max(3,3)=3}{M_2[k]}\right)+1=2+2+2+1=7; \\ b_{\pi_{23}(R),3}&=&\left(\sum_{k=1}^{\max(2,3)=3}{M_3[k]}\right)+1=3+3+1+1=8. \end{eqnarray*}
		\item $ k_E(3,4)=3 \Rightarrow$
		\begin{eqnarray*} b_{\pi_{34}(R),3}&=&\left(\sum_{k=1}^{\max(2,3)=3}{M_3[k]}\right)+1=3+3+1+1=8; \\ b_{\pi_{34}(R),4}&=&\left(\sum_{k=1}^{\max(3,3)=3}{M_4[k]}\right)+1=2+2+2+1=7. \end{eqnarray*}
	\end{itemize}
	
	We have $k_E(2,4)=\infty$ because $M_2=M_4=[2,2,2]$, then to compute ${b}_{\pi_{24}(R)}$ we need to work on the parametrization of $\pi_{2,4}(R)$. We have:
	
	$$ \pi_{2,4}(R)=\mathbb{K}[[(u^2+u^6,w^2+w^9),(u^2+u^7+u^{10},w^2+w^7)]].$$
	
	Both the generators of $\pi_{2,4}(R)$ have valuation $(2,2)$, therefore we have not discrepancies between the orders in the initial parametrization. So we have to produce an element of $\pi_{2,4}(R)$ with discrepancies by manipulating its generators. It suffices to take the difference between them, in fact we find:
	$$\pi_{2,4}(R)  \ni (u^2+u^6,w^2+w^9) - (u^2+u^7+u^{10},w^2+w^7)=(u^6-u^7,-w^7+w^9),$$
	with $\nu((u^6-u^7,-w^7+w^9))=(6,7)$. Because $6=\min(6,7)$ is less or equal than the conductor of $M_2=[2,2,2]$ we can choose $b_{\pi_{24}(R)}=(6+1,6+1)=(7,7)$. 
	
	Finally, denoting with $b[i]$  the bound on the $i$-th branch, we have:
	
	\begin{itemize}
		
		\item $b[1]=\max \left\{ b_{\pi_{12}(R),1}, b_{\pi_{13}(R),1}, b_{\pi_{14}(R),1}\right\}=\max\left\{7,7,7 \right\}=7$;
		\item $b[2]=\max \left\{ b_{\pi_{12}(R),2}, b_{\pi_{23}(R),2}, b_{\pi_{24}(R),2}\right\}=\max\left\{7,7,7 \right\}=7$;
		\item $b[3]=\max \left\{ b_{\pi_{13}(R),3}, b_{\pi_{23}(R),3}, b_{\pi_{34}(R),3}\right\}=\max\left\{7,8,8 \right\}=8$;
		\item $b[4]=\max \left\{ b_{\pi_{14}(R),4}, b_{\pi_{24}(R),4}, b_{\pi_{34}(R),4}\right\}=\max\left\{7,7,7 \right\}=7$.
	\end{itemize}
	
	Then on the $i$-th branch we can truncate all the terms with degree greater than $b[i]$ obtaining the new ring:
	
	$$ S=S_1=\mathbb{K}[[(t^5,u^2+u^6,v^3,w^2),(t^{6},u^2+u^7,v^{7},w^2+w^7)]].$$
	
	Let us show that $S^{*}=R^{*}$.
	
	It is easy to verify that $\pi_{1,2}(S),\pi_{1,3}(S)$ and $\pi_{1,4}(S)$ are all  local. Then for the Lemma  \ref{L2} follows that $\mathfrak{P}(S)=\left\{ \left\{1,2,3,4\right\} \right\}$, in other words $S$ is local.
	
	We have that $\textrm{mult}(S_1)=(5,2,3,2)$. As the minimal value $x_{S_1}$ we can choose $x_1=(t^5,u^2+u^6,v^3,w^2)$.
	
	We have:
	\begin{align*}
	S_2&=\mathbb{K}\left[ \left[(t^5,u^2+u^6,v^3,w^2),\frac{(t^{6},u^2+u^7,v^{7},w^2+w^7)}{x_{S_1}} \right]\right]=\\
	&=\mathbb{K}\left[\left[ (t^5,u^2+u^6,v^3,w^2), \left( t, \frac{1+u^5}{1+u^4},v^4,1+w^5\right) \right]\right]. 
	\end{align*}
	Now we can verify that $\pi_{1,2}(S_2)$ is not local, $\pi_{1,3}(S_2)$ is local, $\pi_{1,4}(S_2)$ is not local and $\pi_{2,4}(S_2)$ is local, therefore $\mathfrak{P}(S_2)=\left\{P_1=\left\{1,3\right\}, P_2=\left\{2,4 \right\}\right\}$.
	We have 
	$$ S_2=S_2^1 \times S_2^2,$$
	
	where $$ S_2^1=\mathbb{K}\left[\left[ (t^5,v^3), \left( t,v^4 \right) \right]\right],$$
	
	\begin{align*}
	S_2^2&=\mathbb{K}\left[\left[ (u^2+u^6,w^2),\left(\frac{1+u^5}{1+u^4},1+w^5\right)\right]\right]=\\
	&=\mathbb{K}\left[\left[ (u^2+u^6,w^2),\left(\frac{-u^4+u^5}{1+u^4},w^5\right)\right]\right].
	\end{align*}
	where, following our conventions on the parametrization, we replace $\left(\frac{1+u^5}{1+u^4},1+w^5\right)$ with  $\left(\frac{1+u^5}{1+u^4},1+w^5\right)-(1,1)=\left(\frac{-u^4+u^5}{1+u^4},w^5\right)$.

	We have $\textrm{mult}(S_2^1)=(1,3)$  and we can choose as element of minimal value the sum $x_{S_2^1}$ of its two generators $$x_{S_2^1}=\left( t+t^5,v^3+v^4\right)$$ while $\textrm{mult}(S_2^2)=(2,2)$ and we can choose as its minimal element $ \displaystyle x_{S_2^2}= \left(u^2+u^6,w^2\right)$.
	Then we have $\textrm{mult}^{*}(S_2)=\left\{ (1,0,3,0), (0,2,0,2) \right\}$ and we can proceed with the computation of $S_3$.
	Thus
	
	$$S_2=\textrm{Bl}(S_2^1) \times \textrm{Bl}(S_2^2),$$
	
	so we have to compute $\textrm{Bl}(S_2^1) $ and $\textrm{Bl}(S_2^2).$

	We have
	$$\textrm{Bl}(S_2^1)=\mathbb{K}\left[\left[\left(\phi_1^{(3)}(t),\psi_1^{(3)}(v)\right),\ldots, \left(\phi_{3}^{(2)}(t),\psi_{3}^{(2)}(v)\right)\right]\right],$$
	
	where
	\begin{itemize}
		
		\item $ \displaystyle \left(\phi_1^{(3)}(t),\psi_1^{(3)}(v)\right)=\left( t+t^5,v^3+v^4\right);$
		
		\item  $ \displaystyle \left(\phi_2^{(3)}(t),\psi_2^{(3)}(v)\right)=\left(\frac{t^4}{1+t^4},\frac{1}{1+v}\right);$
		\item  $ \displaystyle \left(\phi_{3}^{(3)}(t),\psi_{3}^{(3)}(v)\right)=\left(\frac{1}{1+t^4},\frac{v}{1+v}\right).$
	\end{itemize}
	We notice that the second generator has valuation $(4,0)$, then $\textrm{Bl}(S_2^1)$ is not local in $\mathbb{K}[[t]] \times \mathbb{K}[[v]]$. Furthermore we have, with our notation, that $\textrm{mult}^*(\textrm{Bl}(R_2^1))=\left\{(1,0),(0,1)\right\}$.
	Then we have 
	$$ \textrm{Bl}(S_2^1)= \mathbb{K}[[t]] \times \mathbb{K}[[v]].$$
	
	Now we can compute $\textrm{Bl}(S_2^2) $. We have
	$$\textrm{Bl}(S_2^2)=\mathbb{K}\left[\left[ \left(u^2+u^6,w^2\right),\left( \frac{-u^2+u^3}{(1+u^4)^2},w^3\right)\right]\right].$$

	Then we have that $\textrm{Bl}(S_2^2)$ is local in $ \mathbb{K}[[u]] \times \mathbb{K}[w]]$, and  $\textrm{mult}(\textrm{Bl}(S_2^2))=(2,2)$.
	
	Then $\mathfrak{P}(S_3)=\left\{ P_1=\left\{1\right\}, P_2=\left\{ 3\right\}, P_3=\left\{ 2,4\right\}\right\} $ and
	
	$$S_3=\mathbb{K}[[t]] \times \mathbb{K}[[v]] \times S_3^3,$$
	
	with $\textrm{mult}^{*}(S_3)=\left\{ (1,0,0,0), (0,0,1,0) ,(0,2,0,2)\right\}$.
	As a minimal element of $S_3^3$ we can choose again  $x_{S_3^3}=\left(u^2+u^6,w^2\right)$.
	
	Thus $$S_4=\textrm{Bl}(\mathbb{K}[[t]]) \times \textrm{Bl}(\mathbb{K}[[v]]) \times \textrm{Bl}(S_3^3)=\mathbb{K}[[t]] \times \mathbb{K}[[v]] \times \textrm{Bl}(S_3^3).$$
	We have:
	$$ \textrm{Bl}(S_3^3)=\mathbb{K}\left[\left[\left(u^2+u^6,w^2\right),\left( \frac{-1+u}{(1+u^4)^3},w\right)\right]\right].$$
	From this it is easy to show that $\textrm{Bl}(S_3^3)= \mathbb{K}[[u]] \times  \mathbb{K}[[w]]$. 
	Then  $$S_4=\mathbb{K}[[t]] \times  \mathbb{K}[[u]] \times \mathbb{K}[[v]] \times  \mathbb{K}[[w]],$$
	
	and we have reached the stop condition for our algorithm.
	
	We found that $N=4$ and
	\begin{itemize}
		\item $ \textrm{mult}^{*}(S_1)=\left\{ (5,2,3,2) \right\},$  \item $\textrm{mult}^{*}(S_2)=\left\{ (1,0,3,0), (0,2,0,2) \right\}, $ \item $\textrm{mult}^{*}(S_3)=\left\{ (1,0,0,0), (0,0,1,0) ,(0,2,0,2)\right\}, $ \item $\textrm{mult}^{*}(S_4)=\left\{ (1,0,0,0), (0,1,0,0) ,(0,0,1,0),(0,0,0,1)\right\}.$
	\end{itemize}
	
	The corresponding minimal elements are:
	
	\begin{itemize}
		\item $x_{S_1}=(t^5,u^2+u^6,v^3,w^2),$  \item $x^1_{2}=\left( t+t^5,1,v^3+v^4,1\right)$ and $x_2^2= \left(1,u^2+u^6,1,w^2\right);$\item $ x_3^1=(t,1,1,1), x_3^2=(1,1,v,1)$ and $ x_3^3= \left(1,u^2+u^6,1,w^2\right);$ \item $x_4^1=(t,1,1,1), x_4^2=(1,u,1,1), x_4^3=(1,1,v,1)$ and $x_4^4=(1,1,1,w)$.
		
	\end{itemize}
	Then we have the following trees:
	
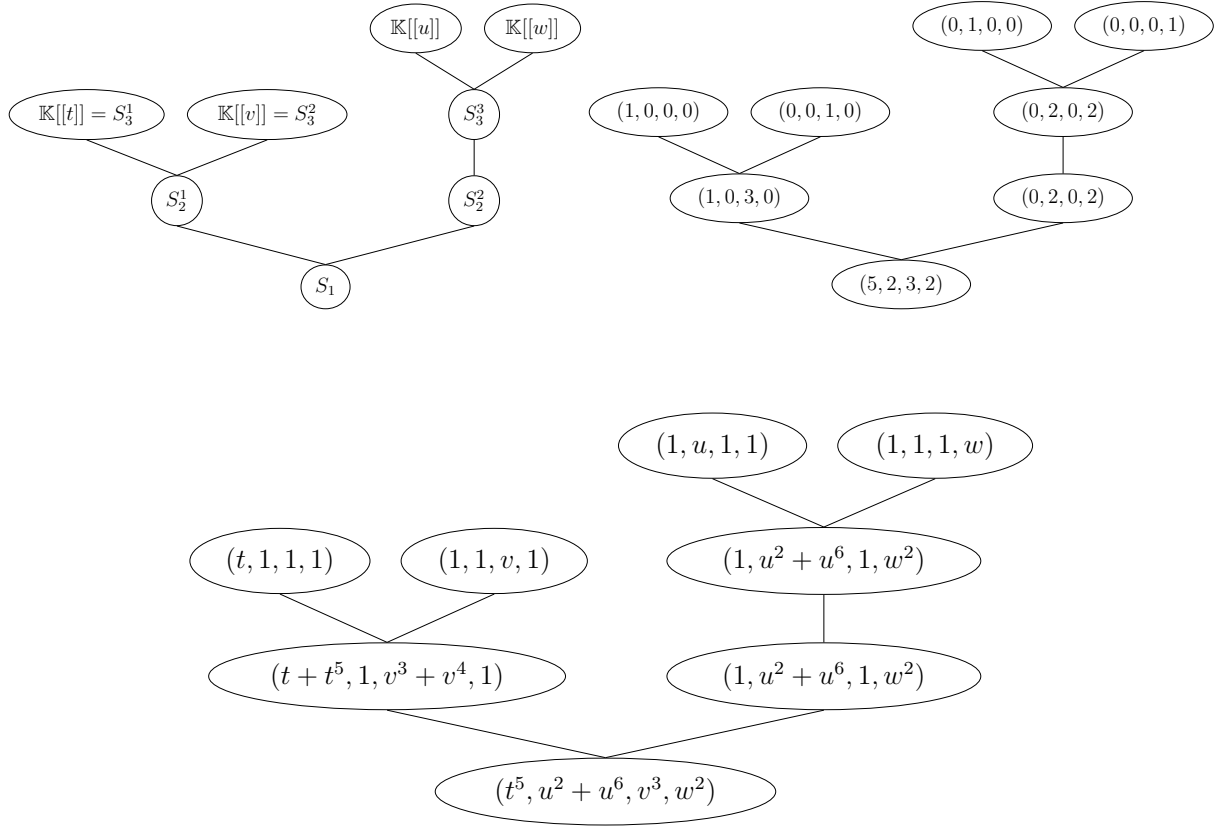
\begin{center}
		\begin{tikzpicture}[grow'=up,sibling distance=13pt,scale=.65]
		\tikzset{level distance=50pt,every tree node/.style={draw,ellipse}} 
		\Tree [ .$S_1$ [ .$S_2^1$ [ .$\K[[t]]=S_3^1$ ] [ .$\K[[v]]=S_3^2$         
		]  ] [.$S_2^2$ [ .$S_3^3$ [ .$\K[[u]]$ ] [ .$\K[[w]]$ ] ] ] ]  \end{tikzpicture} 
		\begin{tikzpicture}[grow'=up,sibling distance=13pt,scale=.65]
		\tikzset{level distance=50pt,every tree node/.style={draw,ellipse}} 
		\Tree [ .$(5,2,3,2)$ [ .$(1,0,3,0)$ [ .$(1,0,0,0)$ ] [ .$(0,0,1,0)$         
		]  ] [.$(0,2,0,2)$ [ .$(0,2,0,2)$ [ .$(0,1,0,0)$ ] [ .$(0,0,0,1)$ ] ] ] ]  \end{tikzpicture} \end{center}
	\vspace{1cm}
	\begin{center}
		\begin{tikzpicture}[grow'=up,sibling distance=13pt,scale=.87]
		\tikzset{level distance=50pt,every tree node/.style={draw,ellipse}} \Tree [ .$(t^5,u^2+u^6,v^3,w^2)$ [ .{$\left(t+t^5,1,v^3+v^4,1\right)$} [ .$(t,1,1,1)$ ] [ .$(1,1,v,1)$         
		]  ] [.$(1,u^2+u^6,1,w^2)$ [ .$(1,u^2+u^6,1,w^2)$ [ .$(1,u,1,1)$ ] [ .$(1,1,1,w)$ ] ] ] ]  \end{tikzpicture} \end{center}
	The conductor of $\nu(S^*)$ is $c=(6,6,6,6)$
	If we compare these tree with the tree computed starting by $R$ in the Example \ref{E2}, we can observe that the tree associated to the ring and the multiplicity tree are the same, instead the minimal tree are equal module $c+1=(7,7,7,7)$.
	Then we have $M(T(S))_E=M(T(R))_E$.\\
	We have that
	\begin{align*}
	&\textrm{Small}(\nu(S^*))=\textrm{Small}(\nu(R^*))=\\
	&=\left\{ ( 5, 2, 3, 2), ( 5, 4, 3,4 ), ( 5, 6, 3, 6 ), ( 6, 2,6, 2), ( 6, 4,6, 4 ), c=( 6, 6,6, 6 ) \right\}.
		\end{align*}
	From the minimal tree we can recover the elements of $S^{*}$ with valuation belonging to $\textrm{Small}(\mathcal{S}(T))$. We can calculate the Arf closure by truncating the terms with degree bigger than the conductor. They are:
	$$\left\{ (t^5,u^2,v^3,w^2), (t^ 5, u^4, v^3, w^4 ), ( t^5, 0, v^3, 0 ), ( 0, u^2, 0, w^2 ), ( 0, u^4, 0, w^4 ) \right\}.$$
	
	Finally we have 
	$$ S^{*}=R^*=\mathbb{K}(1,1,1,1)+\mathbb{K} (t^5,u^2,v^3,w^2)+\mathbb{K}(t^ 5, u^4, v^3, w^4 )+ \mathbb{K}( t^5, 0, v^3, 0 )+$$$$+\mathbb{K}( 0, u^2, 0, w^2 )+\mathbb{K}  ( 0, u^4, 0, w^4 ) +(t^{6},u^6,v^6,w^6) \left(\mathbb{K}[[t]] \times  \mathbb{K}[[u]] \times \mathbb{K}[[v]] \times  \mathbb{K}[[w]]\right).$$
\end{ex}
 \begin{acknowledgements}
	The authors would like to thank Marco D'Anna for his helpful comments and suggestions during the developement of this paper. 
\end{acknowledgements}
\newpage
\printbibliography
\end{document}